\documentclass[11pt]{amsart}

\usepackage[usenames,dvipsnames]{xcolor}
\usepackage{amsmath,amsthm,verbatim,amssymb,amsfonts,amscd, graphicx, enumerate, enumitem,hyperref,dsfont,bbm}
\usepackage{graphics}
\usepackage{tikz-cd}
\usepackage{pinlabel}
\usepackage{stackrel}
\usepackage{mathtools}
\usepackage{wrapfig}
\usepackage[font=small]{caption}
\usepackage{mathrsfs}
\usepackage{extarrows,multirow,colortbl,hhline}\usepackage{amssymb,amsmath,array,diagbox}
\usepackage{parskip}
\usepackage{blkarray}
\usepackage[numbers]{natbib}
\setcitestyle{open={[},close={]}}

\definecolor{linkpurple}{RGB}{136,92,150}
\definecolor{cellgray}{gray}{0.875}

\hypersetup{colorlinks=true,linkcolor=MidnightBlue,citecolor=MidnightBlue, urlcolor=black}

\usepackage[margin=1.3in,headsep=.3in, top=1.2in, bottom=1.15in]{geometry}

\newtheoremstyle{plain}
  {5pt plus3pt minus3pt}   % ABOVESPACE
  {5pt plus3.3pt minus3.3pt}   % BELOWSPACE
  {\itshape}  % BODYFONT
  {0pt}       % INDENT (empty value is the same as 0pt)
  {\bfseries} % HEADFONT
  {.}         % HEADPUNCT
  {5pt plus 1pt minus 1pt} % HEADSPACE
  {}          % CUSTOM-HEAD-SPEC

\newtheorem{theorem}{Theorem}[section]

\newtheorem{lemma}[theorem]{Lemma}
\newtheorem{proposition}[theorem]{Proposition}
 
\newtheorem{question}[theorem]{Question}
\newtheorem*{question*}{Question} 
 
\newtheorem*{wall*}{Wall's Stabilization Problem}

\newtheorem{mainthm}{Theorem}[section]

\theoremstyle{definition}

\newtheoremstyle{example}
  {2pt plus3.3pt minus3.6pt}   % ABOVESPACE
  {2pt plus3.3pt minus3.3pt}   % BELOWSPACE
  {}  % BODYFONT
  {0pt}       % INDENT (empty value is the same as 0pt)
  {\bfseries} % HEADFONT
  {.}         % HEADPUNCT
  {5pt plus 1pt minus 1pt} % HEADSPACE
  {}          % CUSTOM-HEAD-SPEC
\theoremstyle{example}
\newtheorem{example}[theorem]{Example}

\theoremstyle{definition}

\newtheoremstyle{rmk}
  {0pt plus3.6pt minus6.6pt}   % ABOVESPACE
  {1pt plus3.3pt minus3.3pt}   % BELOWSPACE
  {}  % BODYFONT
  {0pt}       % INDENT (empty value is the same as 0pt)
  {\bfseries} % HEADFONT
  {.}         % HEADPUNCT
  {5pt plus 1pt minus 1pt} % HEADSPACE
  {}          % CUSTOM-HEAD-SPEC
\theoremstyle{rmk}
\newtheorem{remark}[theorem]{Remark}

\newcommand{\leftrarrows}{\mathrel{\raise.75ex\hbox{\oalign{%
  $\scriptstyle\leftarrow$\cr
  \vrule width0pt height.5ex$\hfil\scriptstyle\relbar$\cr}}}}
\newcommand{\lrightarrows}{\mathrel{\raise.75ex\hbox{\oalign{%
  $\scriptstyle\relbar$\hfil\cr
  $\scriptstyle\vrule width0pt height.5ex\smash\rightarrow$\cr}}}}
\newcommand{\Rrelbar}{\mathrel{\raise.75ex\hbox{\oalign{%
  $\scriptstyle\relbar$\cr
  \vrule width0pt height.5ex$\scriptstyle\relbar$}}}}

\usepackage[utf8]{inputenc}

\makeatletter
\def\leftrightarrowsfill@{\arrowfill@\leftrarrows\Rrelbar\lrightarrows}
\newcommand{\xleftrightarrows}[2][]{\ext@arrow 3399\leftrightarrowsfill@{#1}{#2}}
\makeatother

\definecolor{MutedBlue}{RGB}{40,90,160}
\definecolor{violet}{rgb}{.6,0,.6}
\definecolor{green}{rgb}{.0,.8,0}
\definecolor{darkred}{rgb}{.7,0,0}
\definecolor{darkyellow}{rgb}{.4,.4,0}
\definecolor{orangenew}{rgb}{.69,.37,.25}
\definecolor{MutedRed}{RGB}{200,55,55}

\newcommand{\Z}{\mathbb{Z}}

\newcommand{\kk}{\mathbb{F}}%{\mathds{k}}

\newcommand{\F}{\mathbb{F}}

\let\int\relax
\newcommand{\int}{\mathring}

\DeclareMathOperator{\ckhr}{\widetilde{CKh}}

\DeclareMathOperator{\khr}{\widetilde{Kh}}

\DeclareMathOperator{\bnr}{\widetilde{BN}}

\DeclareMathOperator{\cbnr}{\widetilde{CBN}}

\DeclareMathOperator{\hfk}{\mathit{HFK}}

\newcommand{\sxs}{S^2 \! \times \! S^2}
\newcommand{\smallsum}{\raisebox{1pt}{\text{\footnotesize$\#$}}}

\usepackage{fancyhdr}
\usepackage[affil-it]{authblk} 
\usepackage{etoolbox}
\usepackage{lmodern}
\makeatletter
\patchcmd{\@maketitle}{\LARGE \@title}{\fontsize{16}{19.2}\selectfont\@title}{}{}
\makeatother

\author[K.\ Hayden]{Kyle Hayden}
\address{Rutgers University, Newark, NJ 07102}
\email{kyle.hayden@rutgers.edu}

\title[An atomic approach to stabilization problems]{\large An atomic approach to Wall-type stabilization problems}

\begin{document}
\vspace*{-.375in}

\bigskip
%\maketitle

\begingroup
\def\uppercasenonmath#1{}
\let\MakeUppercase\relax
\maketitle
\endgroup
\thispagestyle{empty}

\vspace{-.25in}

\begin{center}\small
\textsc{Kyle Hayden}
\end{center}

%\vspace{-.15in}

\bigskip

\begin{center} \begin{minipage}{.88\linewidth}\footnotesize

\textsc{Abstract.} Wall-type stabilization problems investigate the collapse of exotic 4-dimensional phenomena under stabilization operations (e.g., taking connected sums with $\sxs$). We propose an elementary approach to these problems, providing a construction of  exotic \linebreak {4-manifolds}  and knotted  surfaces  that are candidates to remain exotic after stabilization --- including examples in the setting of closed, simply connected 4-manifolds.   As a proof of concept, we show this construction yields exotic surfaces in $B^4$ that remain exotic after  (internal) stabilization, detected by the cobordism maps on universal Khovanov homology.
 
\hspace{1em} We also compare these  Khovanov-theoretic obstructions for  surfaces to the Floer-theoretic counterparts for    exotic 4-manifolds obtained as their branched covers, suggesting a bridge via  Lin's spectral sequence from Bar-Natan homology to involutive monopole Floer homology.

\end{minipage}
\end{center}

\bigskip

\section{Introduction}\label{sec:intro}

\smallskip

In \cite{wall:4-manifolds}, Wall proved that any pair of 
smooth, closed, simply-connected 4-manifolds that are homotopy equivalent   are  $h$-cobordant and become diffeomorphic after taking connected sums with $k(S^2 \! \times \! S^2)$ for some $k \geq0$. The  number of stabilizations required, $k$, conveys a measure of complexity for the $h$-cobordism and thus a notion  of distance between  smooth structures.  Donaldson's disproof of the smooth $h$-cobordism conjecture produced exotic pairs of closed, simply-connected 4-manifolds  \cite{donaldson:h-cobordism}, implying one must allow $k >0$. In turn, understanding the possible values of $k$ helps quantify the breakdown of the $h$-cobordism theorem in dimension four. 

\vspace{-4pt}

\begin{wall*}
For $k \geq 1$, does there exist an exotic pair of closed, simply-connected 4-manifolds that remain distinct after connected sum with $k(S^2 \!\times \!S^2)$?
\end{wall*}

\vspace{-4pt}

 Subsequent work has showed that Wall's principle --- the  instability of exotic phenomena  under connected sum with $\smash{S^2\! \times \! S^2}$  
--- holds in a variety of settings, including for all compact, orientable smooth 4-manifolds \cite{gompf:stable} and  exotically knotted surfaces and self-diffeomorphisms \cite{quinn:isotopy,perron2}. For knotted surfaces, Baykur-Sunukjian \cite{baykur-sunukjian:stab} also established a Wall-type result for the  related notion of \emph{internal stabilization}, which consists of adding an embedded (and possibly knotted) handle to increase the genus of a knotted surface.  Recently, there has been a burst of progress demonstrating that one stabilization is \emph{not} enough in many contexts \cite{lin:twist,lin-mukherjee,konno-mukherjee-taniguchi,guth,kang}, as well as results identifying  contexts in which a single stabilization suffices   \cite{akmrs:one-is-enough,auckly-sadykov,ruberman-strle}. 

This paper proposes a strategy for building exotic pairs of 4-manifolds and knotted surfaces that are candidates to remain distinct after stabilization,  including   closed, simply-connected  4-manifolds (Theorem~\ref{thm:embed}). Our approach is guided by certain building blocks that  should arise in $h$-cobordisms with sufficiently complicated sets of critical points.   As a proof of concept, we  produce exotically knotted  surfaces  in $B^4$ that remain distinct after one  internal stabilization, detected using the universal version of Khovanov homology.

\vspace{-4pt}

\begin{mainthm}\label{thm:stab}
For all integers $g \geq 1$, there are %infinitely many 
exotically knotted pairs of properly embedded, genus-$g$ surfaces in $B^4$ that remain exotic after one  internal stabilization, where the stabilized surfaces induce distinct maps on universal Khovanov homology.
\end{mainthm}

\vspace{-4pt}

Below we sketch our constructions, obstructions, and potential connections to Floer theory.

\vspace{-.05in}

\subsection{The topological construction}\label{subsec:build} 
Inspired by the theory of corks \cite{akbulut:cork,CFHS,matveyev,kirby:cork} and Gompf's nuclei of elliptic surfaces \cite{gompf:nuclei}, we aim to identify the core 4-dimensional building blocks that contribute to the complexity of an $h$-cobordism --- then we build our 4-manifolds around these atomic pieces; see \S\ref{sec:top}.  Here the cork theorem says that if $X$ and $X'$ form an exotic pair of closed, simply-connected 4-manifolds, then there is a compact, contractible 4-manifold $C \subset X$ and an involution of $\partial C$ such that removing $\mathring{C}$ from $X$ and regluing it by the involution of $\partial C$ yields $X'$.

 \begin{wrapfigure}[36]{r}{0.205\textwidth}
\center
\vspace{-20pt}
\def\svgwidth{.95\linewidth}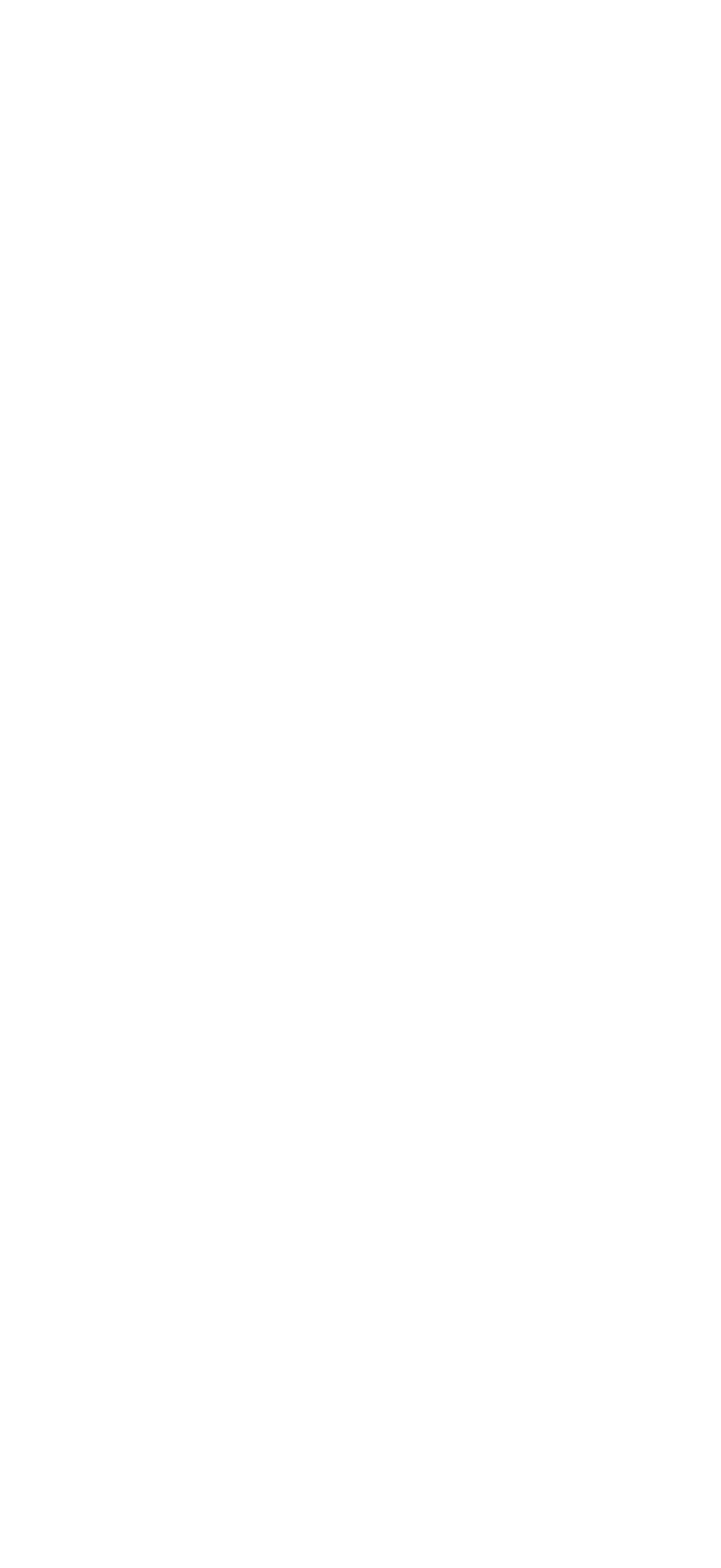\hfill

\vspace{-4pt}
\caption{}\label{fig:hooked-positron}

\vspace{10pt}

\def\svgwidth{.95\linewidth}%% Creator: Inkscape 1.2 (dc2aeda, 2022-05-15), www.inkscape.org
%% Accompanies image file 'hooked-positron-disks-vert-squish.pdf' (pdf, eps, ps)
\begingroup%
  \makeatletter%
  \providecommand\color[2][]{%
    \errmessage{(Inkscape) Color is used for the text in Inkscape, but the package 'color.sty' is not loaded}%
    \renewcommand\color[2][]{}%
  }%
  \providecommand\transparent[1]{%
    \errmessage{(Inkscape) Transparency is used (non-zero) for the text in Inkscape, but the package 'transparent.sty' is not loaded}%
    \renewcommand\transparent[1]{}%
  }%
  \providecommand\rotatebox[2]{#2}%
  \newcommand*\fsize{\dimexpr\f@size pt\relax}%
  \newcommand*\lineheight[1]{\fontsize{\fsize}{#1\fsize}\selectfont}%
  \ifx\svgwidth\undefined%
    \setlength{\unitlength}{267.92220853bp}%
    \ifx\svgscale\undefined%
      \relax%
    \else%
      \setlength{\unitlength}{\unitlength * \real{\svgscale}}%
    \fi%
  \else%
    \setlength{\unitlength}{\svgwidth}%
  \fi%
  \global\let\svgwidth\undefined%
  \global\let\svgscale\undefined%
  \makeatother%
  \begin{picture}(1,3.66048931)%
    \lineheight{1}%
    \setlength\tabcolsep{0pt}%
    \put(0,0){\includegraphics[width=\unitlength,page=1]{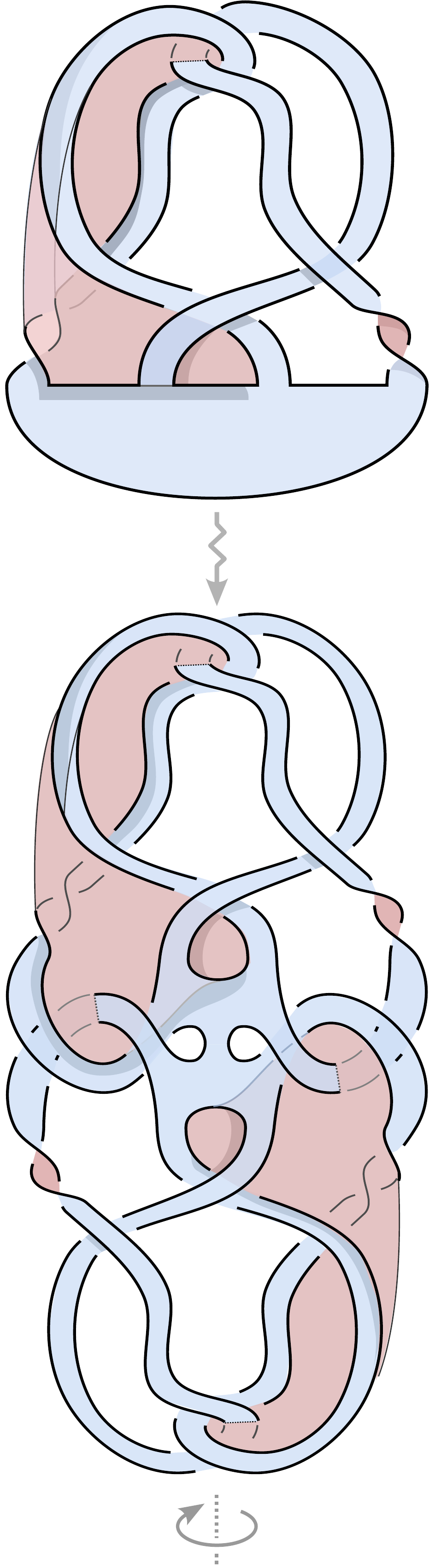}}%
    \put(0.65480355,0.06722068){\color[rgb]{0.6,0.6,0.6}\makebox(0,0)[lt]{\lineheight{1.25}\smash{\begin{tabular}[t]{l}$\tau$\end{tabular}}}}%
  \end{picture}%
\endgroup%

\vspace{-7.5pt}
\caption{}\label{fig:hooked-disks}

\end{wrapfigure}
 \

\vspace{-20pt}

This  modification along $C \subset X$ can also be achieved by a sequence of surgeries on $C$ that has the effect of exchanging 1- and 2-handles in algebraically dual pairs. The number of such pairs in $C$ gives a naive upper bound on the number of stabilizations required to dissolve the cork twist.  However, if there is insufficient geometric linking between the different pairs of handles,  a copy of $S^2 \! \times \! S^2$ can be recycled and used to effect multiple 1-/2-handle exchanges.  We attempt to prohibit this recycling  by taking multiple corks and  interlocking their various pairs of 1- and 2-handles.   (In \S\ref{sec:top}, we recast this in the language of $h$-cobordisms and configurations of  2-spheres.)

\begin{example}\label{ex:plug} The top of Figure~\ref{fig:hooked-positron} depicts  the \emph{positron cork} from \cite{akbulut-matveyev:decomp}, where the cork twisting involution acts by  180$^\circ$ rotation through a vertical axis. Below it, we build a 4-manifold $W$  from two interlocked positron corks. Although $W$ is not itself a cork (because it is not contractible, having $H_1 \cong \Z/5\Z$),  it is a \emph{plug} in the sense of \cite{akbulut-yasui:plugs} and can be used to build larger exotic 4-manifolds. \hfill $\diamond$\end{example}

By varying the linking and the  number of 1-/2-handle pairs, one can adjust the algebraic topology (e.g., achieving contractibility)  and increase the expected number of stabilizations required to dissolve the twist. Moreover, the explicit nature of the construction makes it straightforward to embed these compact 4-manifolds with boundary into closed 4-manifolds to produce candidates in the closed setting.

\begin{mainthm}\label{thm:embed}
There exists an exotic pair of closed  4-manifolds related by twisting along the 4-manifold $W$ depicted in Figure~\ref{fig:hooked-positron}. 
\end{mainthm}

  A similar  construction applies to knotted surfaces; we fix a handle structure on the surface exterior and swap algebraically dual pairs of 1- and 2-handles. Equivalently, we swap \emph{pairs} of 0-handles in the surface (i.e., disks) with 1-handles (i.e., bands) through a sequence of decompressions and compressions. By  interlocking copies of the surface, we   stop these swaps from being performed one at a time.
  
\begin{example}\label{ex:disk} The top of Figure~\ref{fig:hooked-disks} depicts the \emph{positron disk}, whose double branched cover is the positron cork \cite{hayden:curves}. Two copies  are interlocked and merged to yield a slice disk $D$ whose double branched cover is $W$ from Example~\ref{ex:plug}. Applying the rotation $\tau$ yields a disk $D'$ with $\partial D'=\partial D$; these are the  building blocks for Theorem~\ref{thm:stab}. \hfill $\diamond$
\end{example}

To build the exotic surfaces in Theorem~\ref{thm:stab}, we  attach additional bands along the knot $K=\partial D = \partial D'$ in a manner designed to annhilate the obstruction to topological isotopy while preserving the smooth obstruction from Khovanov homology. We prove that the resulting genus-1 surfaces are topologically isotopic using results of Conway-Powell \cite{conway-powell} (Theorem~\ref{thm:cp} below); the key technical step is a calculation of  equivariant intersection forms. Higher-genus examples  are then obtained through boundary-connected sums. %The obstructive half of our argument is sketched below in \S\ref{subsec:obstructions}.

The natural question is whether exotic 4-manifolds such as those constructed in Theorem~\ref{thm:embed} remain distinct after stabilization with $S^2 \! \times \! S^2$.  The proof of Theorem~\ref{thm:stab} offers positive evidence because (standard) internal stabilization of a surface corresponds to external stabilization of its double branched cover. In \S\ref{subsec:compare}, we further examine this connection and suggest a bridge via Lin's spectral sequence relating Bar~Natan homology and involutive monopole Floer homology \cite{lin:bar-natan}.

\begin{remark}
Several of the recent breakthroughs mentioned above occurred while this project was in progress (cf \cite{hayden:nsf}), so we draw some comparisons at the end of  \S\ref{subsec:compare}. In particular, Theorem~\ref{thm:stab} is most  comparable to work of Guth \cite{guth}, who used cabling operations to produce exotic surfaces (including disks) in $B^4$ that remain distinct after any prescribed number of internal stabilizations.
\end{remark}

\smallskip

\subsection{Obstructions from Bar-Natan homology.} \label{subsec:obstructions} 
The cobordism maps in Khovanov's original theory have proven to be effective at distinguishing knotted surfaces in $B^4$ \cite{sundberg-swann,hayden-sundberg,hkmps}, and they appear well-suited to  cases where explicit Floer-theoretic calculations are prohibitively complex. The   proof of Theorem~\ref{thm:stab} relies on the additional structure available in universal versions of Khovanov homology over $\kk[H]$, where $\kk$ is a field and $H$ is a formal variable. Our approach uses the notion of $H$-torsion order in (reduced) Bar-Natan homology $\bnr$, analogs of which have previously been used to study unknotting number and ribbon distance in \cite{alishahi:unknotting,alishahi-dowlin,sarkar:ribbon,gujral:ribbon,ilm:unknotting}.

We sketch our approach to Theorem~\ref{thm:stab}, continuing with the disks $D$ and $D'$ as a running example: For technical reasons, it is convenient to consider the mirrored, time-reversed disks $-D$ and $-D'$, which induce cobordism maps $\bnr(-D)$ and $\bnr(-D')$. The images of these maps lie in $\bnr(-K)$ and are generated by elements $\bnr(-D)(1)$ and $\bnr(-D')(1)$. Distinguishing the underlying disks $-D$ and $-D'$ amounts to showing that the difference 

\vspace{-7.5pt}

$$\delta = \bnr(-D)(1) - \bnr(-D')(1)$$ 

is a nonzero element in $\bnr(-K)$.  To do so, we perform the corresponding calculation in Khovanov's original reduced theory and then  lift it to Bar~Natan homology over $\kk[H]$. 

Next we turn to the \emph{stabilized} surfaces. The effect of internal stabilization is to multiply these cobordism maps by the variable  $H$ in $\kk[H]$ (cf Lemma~\ref{lem:stab}). Thus we must further prove that $H \cdot \delta$ is nonzero.\footnote{For the analogous $X$-action in Lee homology, Sarkar says such elements have \emph{nontrivial extortion order}.} This is achieved with the assistance of computer calculations that enable us to prove that \emph{all} nonzero elements of $\bnr(-K)$ with the same bigrading as $\delta$ survive multiplication by $H$, completing the argument.\footnote{For this $H$-action on Bar~Natan homology, we say this bigrading is a \emph{nontrivial torsion hoarder}.}

\subsection{Connections and comparisons with Floer homology.}\label{subsec:compare} In \cite{os:branched}, Ozsv\'ath-Szab\'o  established a spectral sequence relating the reduced Khovanov homology $\smash{\khr(L)}$ of a link $L\subset S^3$ to the Heegaard Floer homology of its mirror's double branched cover $\smash{\Sigma(-L)=\Sigma_2(S^3,-L)}$. Subsequent work has established a variety of similar spectral sequences relating Khovanov-type link homologies to Floer-type theories. In particular,  Lin \cite{lin:bar-natan} established a spectral sequence from a truncated version of a link's Bar-Natan homology over $\F_2[H]/H^2$, denoted $\smash{\bnr_2(L)}$, to the involutive monopole Floer homology of the  double branched cover $\Sigma_2(S^3,-L)$ over $\F_2[Q]/Q^2$, which we denote by $\smash{\widetilde{HMI}(\Sigma(-L))}$. Moreover, after setting $H=Q$, this is a spectral sequence of $\F_2[Q]/Q^2$-modules. 

The argument that distinguishes the cobordism maps on $\bnr(\, \cdot \,)$ over $\F_2[H]$ induced by stabilized surfaces   in the proof of Theorem~\ref{thm:stab} easily adapts to prove that these stabilized surfaces induce distinct maps in the truncated theory $\bnr_2(\, \cdot \,)$ over $\F_2[H]/H^2$. This hints at the possibility of using involutive monopole Floer homology to prove that the stabilized plug $W \smallsum (S^2 \! \times \! S^2)= \Sigma_2(B^4,D\smallsum T^2)$ remains  nontrivial.

\vspace{-4pt}

\begin{question}
Does the class $\delta_2 \in \bnr_2(-K)$ given by the difference of the disks' induced maps survive to the page $E^\infty \cong \widetilde{HMI}(\Sigma(K))$?
\end{question}

\vspace{-5pt}

We expect that the class $\delta_2 \in \bnr_2(-K)$ has a counterpart in  $\widetilde{HMI}(\Sigma(K))$, perhaps given by the difference of two particular contact invariants in monopole Floer homology. In particular, $W$ admits a Stein structure (as depicted in Figure~\ref{fig:positron-stein}) that induces a contact structure $\xi$ on $\partial W$, and the arguments in the proof of Theorem~\ref{thm:embed} imply that $\xi$ and $\xi'=\tilde\tau_*( \xi)$ induce distinct contact elements in monopole Floer homology. If this difference survives multiplication by $Q$, the stabilized plug is nontrivial.

\smallskip

\begin{remark} \textbf{(a)} 
We do not necessarily expect the class $\delta_2 \in \bnr_2(-K)$ to have a contact-geometric interpretation, e.g., as the difference of Plamenevskaya-type invariants \cite{plamenevskaya:transverse-Kh,lns:transverse} of two transverse representatives of $K$. 
There \emph{is} indeed a natural transverse representative $\mathcal{K}$ of $K$  which arises from realizing $D$ and $D'$ as pieces of complex curves in $B^4 \subset \mathbb{C}^2$ with transversely isotopic boundary in $S^3$; see Figure~\ref{fig:positron-stein}. 
However, the induced contact structure on $\Sigma(K)$ is $\tilde\tau$-equivariant up to isotopy, hence cannot be %isotopic to 
$\xi$ or $\xi'$. 

\vspace{-4.5pt}

\quad \ \textbf{(b)} The  elementary 4-manifolds underlying our  construction in \S\ref{sec:top} arise naturally when studying 4-dimensional $h$-cobordisms (cf \cite{kirby:cork}). In \cite{ladu}, Ladu studies the monopole Floer homology of such 4-manifolds, dubbed ``protocorks''; this may be a starting point towards understanding the involutive monopole Floer homology of our 4-manifolds. 
\end{remark}

\smallskip

Finally, we compare our examples and those in \cite{guth}. Guth's primary examples are obtained by taking $(p,1)$-cables of the positron disk (top of Figure~\ref{fig:hooked-disks}) for $p \geq 2$. For knot Floer homology $\hfk^-$ over $\F_2[U]$, internal stabilization acts as multiplication by $U$, and Guth proves that his surfaces induce cobordism maps whose images are generated by elements in $\hfk^-$ whose difference has $U$-torsion order $p-1$. This approach has several strengths, including the ability to detect large stabilization distance and the flexibility to be adapted to cables of many  pairs of disks \cite{guth-hayden-kang-park}. Our Khovanov-theoretic obstructions do not readily apply to such disks. For example, the Bar-Natan homology of the $(2,1)$-cable of the positron knot has $H$-torsion order equal to 1, the same as the positron knot itself. On the other hand,  the branched covers of Guth's stabilized surfaces are known to be diffeomorphic \cite{guth-hayden-kang-park}; those arguments do not apply to our examples.

An optimistic take may be that these differences are not weaknesses of Khovanov-type invariants, but rather that torsion order in Khovanov homology measures different  phenomena than $U$-torsion order in $\hfk^-$. Indeed, Lin's spectral sequence suggests $H$-torsion in Bar-Natan homology is analogous to the $Q$-torsion in involutive Floer theories, such as the $Q$-torsion used in Kang's recent proof that one stabilization is not enough for contractible 4-manifolds \cite{kang}.  This prediction aligns with the relative scarcity of $H$-torsion in Khovanov homology  (as evidenced by the difficulty in finding counterexamples to the Knight Move Conjecture \cite{manolescu-marengon} and by the weaker bounds on unknotting number and ribbon distance obtained in \cite{alishahi:unknotting,alishahi-dowlin,sarkar:ribbon,gujral:ribbon,ilm:unknotting} as compared to, for example, \cite{jmz:torsion}).   

\smallskip

\subsection*{Acknowledgements} The author thanks John Baldwin, Onkar Gujral, Mikhail Khovanov, Siddhi Krishna, Lukas Lewark, Ciprian Manolescu, Lisa Piccirillo, Isaac Sundberg, and Melissa Zhang for  suggestions, sanity checks, and stimulating conversations. Thanks also to BIRS-CMO  and the organizers of the workshop 
\textsl{Using Quantum Invariants to Do Interesting Topology} (22w5171), where the final piece of this project fell into place. The experimental phase of this project benefitted from   the Kirby Calculator \cite{KLO} and \texttt{khoca} \cite{khoca}. Additional support from NSF grant DMS-2243128.

\medskip
\section{The construction}\label{sec:top}

%\vspace{1pt}

\subsection{Stabilization, $\boldsymbol{h}$-cobordisms,  and the recycling problem}
%We begin by recalling the anatomy of  an $h$-cobordism $M$ between simply connected 4-manifolds $X_0$ and $X_1$ (cf \cite{kirby:cork}). As discussed in \cite[p147]{wall:4-manifolds}, standard arguments allow us to choose a handle structure on $M$ consisting only of 2- and 3-handles. Moreover, we may arrange for these to arise in algebraically (but not geometrically) canceling 2-/3-handle pairs, and we let $n$ denote the number of such pairs.

We begin by recalling that any $h$-cobordism $M$ between simply connected 4-manifolds $X_0$ and $X_1$  admits a handle structure consisting only of 2- and 3-handles (cf \cite{wall:4-manifolds,kirby:cork}). Moreover, we may arrange for these to arise in algebraically (but not geometrically) canceling 2-/3-handle pairs, and we let $n$ denote the number of such pairs.

 After attaching the 5-dimensional 2-handles to $X_0 \times [0,\epsilon]$, we obtain an intermediate 4-manifold $X_{1/2} \subset M$ that can be obtained from $X_0$ by $n$ stabilizations with  $\sxs$. Indeed, attaching a 5-dimensional 2-handle has the 4-dimensional effect of surgery on a loop. Since $X_0$ is simply connected, the loop is trivial and the surgery results in a stabilization with  $\sxs$ or a \emph{twisted} stabilization with $S^2 \tilde \times S^2$, but here the latter can be ruled out \cite[p147]{wall:4-manifolds}. Turning  $M$ upside down reverses the roles of 2- and 3-handles, so an analogous argument shows  that $X_{1/2}$ can also be obtained from $X_1$ by $n$ stabilizations. It follows that $X_0\, \smallsum\, n(\sxs) \cong X_1 \, \smallsum \, n (\sxs)$.  However, this naive upper bound on the stabilization number  can often be reduced by ``recycling'' an $\sxs$-summand, as below.

 \begin{wrapfigure}[10]{r}{0.18\textwidth}
\center
\vspace{-10.5pt}
\vspace{2pt}
 \def\svgwidth{.95\linewidth}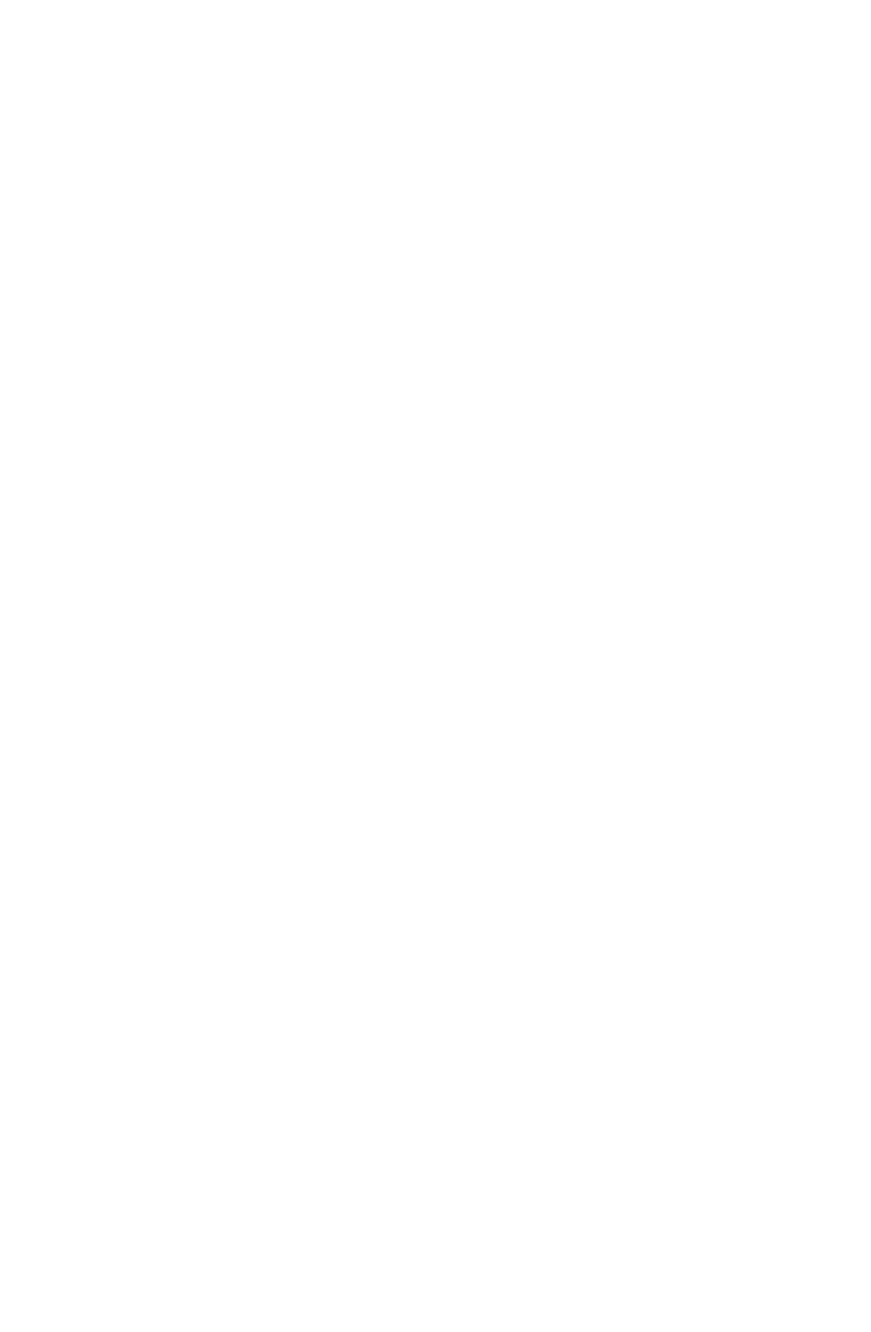

\vspace{-9.5pt}
\caption{}\label{fig:recycling-opportunity}
\end{wrapfigure}

\

\vspace{-20pt}

\begin{example} \textbf{(a)} As shown below in Figure~\ref{fig:recycle}, the positron cork becomes diffeomorphic (rel boundary) to its twist after one stabilization. If  exotic 4-manifolds $X_0$ and $X_1$ are related by twisting along two disjoint copies of the positron cork in $X_0$, then there is an $h$-cobordism from $X_0$ to $X_1$ with two 2-/3-handle pairs, implying  $X_0 \smallsum 2(\sxs) \cong X_1 \smallsum 2(\sxs)$. But because the corks are disjoint, we may use a single $\sxs$ summand for both corks, implying $X_0 \, \smallsum\, \sxs \cong X_1\, \smallsum \, \sxs$. 

\vspace{-4pt}

\quad \ \textbf{(b)} For a more subtle example, the reader  is encouraged to show that one stabilization trivializes the cork in Figure~\ref{fig:recycling-opportunity}. \end{example}

\begin{figure}\center

\def\svgwidth{.875\linewidth}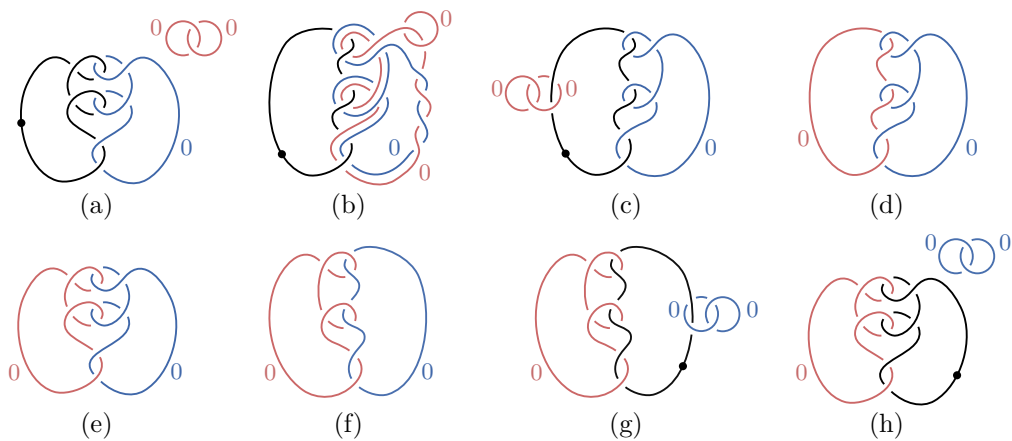
\captionsetup{width=.9\linewidth}

\caption{The positron cork twist after stabilization with $S^2  \times S^2$.  Part (b) is obtained from (a) by a handle slide. Part (c) is obtained by using handle slides over the small 0-framed meridional 2-handle to effect crossing changes. Parts (d)-(f) are related by isotopy. Parts (g)-(h) are obtained similarly to (a)-(d), in reverse.}\label{fig:recycle}

\vspace{-5pt}
\end{figure}

%\vspace{-3pt}

Returning to the  setup above, observe that $X_{1/2}$ contains two collections of disjoint, smoothly embedded 2-spheres $\mathcal{A}=\{A_1,\ldots,A_n\}$ and $\mathcal{B}=\{B_1,\ldots,B_n\}$ given by the attaching spheres for the 3-handles and the belt spheres of the 2-handles in $M$, respectively. (We assume these spheres are transverse to one another.) Recycling opportunities arise when the configuration of spheres $\mathcal{A} \cup \mathcal{B}$ is insufficiently complicated, allowing us to reorder handles so that some 3-handles appear before all 2-handles have been attached.

\subsection{Holding handles together}\label{subsec:holding} To produce pairs of 4-manifolds that are candidates to remain distinct after stabilization, we  \emph{begin} with a neighborhood of a sufficiently complicated configuration of spheres $\mathcal{A} \cup \mathcal{B}$ (much as in the proof of the cork theorem \cite{kirby:cork}), and then enlarge it by attaching 4-dimensional handles to yield a 4-manifold $X_{1/2}$. By attaching 5-dimensional 3-handles to $X_{1/2} \! \times \! [-\epsilon,\epsilon]$ along $\mathcal{B} \subset X_{1/2}  \! \times \! \{-\epsilon\}$ and $\mathcal{A} \subset X_{1/2} \! \times\! \{\epsilon\}$, we produce a cobordism between two 4-manifolds $X_0$ and $X_1$. These 4-manifolds are obtained by surgering $X_{1/2}$ along the spheres in $\mathcal{B}$ and $\mathcal{A}$, respectively.

\begin{remark} Neighborhoods of these sphere configurations are plumbings of disk bundles over spheres; their handle diagrams can be produced using  \cite[\S6.1]{gompfstipsicz}.  At the level of handle diagrams, $X_0$ and $X_1$ are obtained from $X_{1/2}$ by converting the 0-framed 2-handles associated to $\mathcal{B}$ and $\mathcal{A}$, respectively, into dotted 1-handle curves. 
\end{remark}

%\vspace{-5pt}

\begin{example} A schematic depiction of one such configuration of 2-spheres is shown on the left side of Figure~\ref{fig:spheres}. (Here however, we have \emph{not} arranged that $A_i \cdot B_j = \delta_{ij}$.)  On the right side of Figure~\ref{fig:spheres} is a handle diagram for a neighborhood $N$ of the configuration of spheres. The 4-manifold $W$ from Example~\ref{ex:plug} is obtained from $N$ by  attaching a mix of 0- and (-1)-framed 2-handles along meridians to the dotted 1-handle curves in Figure~\ref{fig:spheres}, followed by surgery on the 2-spheres $B_1$ and $B_2$.
\end{example}

\begin{figure}\center

\def\svgwidth{\linewidth}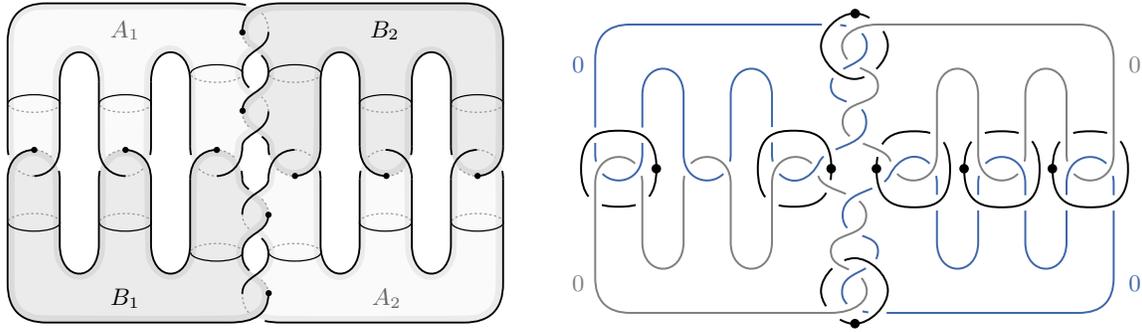

\caption{On the left, a schematic depicting a configuration of intersecting 2-spheres. On the right, a handle diagram for a neighborhood of this configuration.}\label{fig:spheres}

%\vspace{-3pt}
\end{figure}

%\vspace{-5pt}

We point out that, without ensuring the spheres in  $\mathcal{A}$ and $\mathcal{B}$ satisfy $A_i \cdot B_j = \delta_{ij}$, the 5-dimensional cobordism constructed above need not even be a homology cobordism. However, relaxing these conditions provides a larger class of useful 4-manifolds. 

\subsection{Knotted surfaces} For a knotted surface $\Sigma$ in a 4-manifold $X$, we may consider the effect of  \emph{external stabilization} (i.e., including $\Sigma$ into $X\smallsum \sxs$) or \emph{internal stabilization}. As described in \cite[\S2.1]{baykur-sunukjian:stab}, the latter consists of choosing an embedded 3-dimensional 1-handle $h \approx [-1,1] \! \times \! D^2$ in $X$ that intersects  $\Sigma$ only along $ \{\pm1\} \times D^2$, then removing $ \{\pm1\} \times \mathring{D}^2$ from $\Sigma$ and gluing in $[-1,1]\times \partial D^2$. In this paper, we will only consider internal stabilizations that preserve orientability.  As shown in \cite{baykur-sunukjian:stab}, any pair of homologous surfaces $\Sigma_0$ and $\Sigma_1$ of equal genus in a 4-manifold $X$ become smoothly isotopic after $n$ internal stabilizations for some $n \geq 0$.

 \begin{wrapfigure}[16]{r}{0.235\textwidth}
\center
\vspace{-11.5pt}
\includegraphics[width=.96\linewidth]{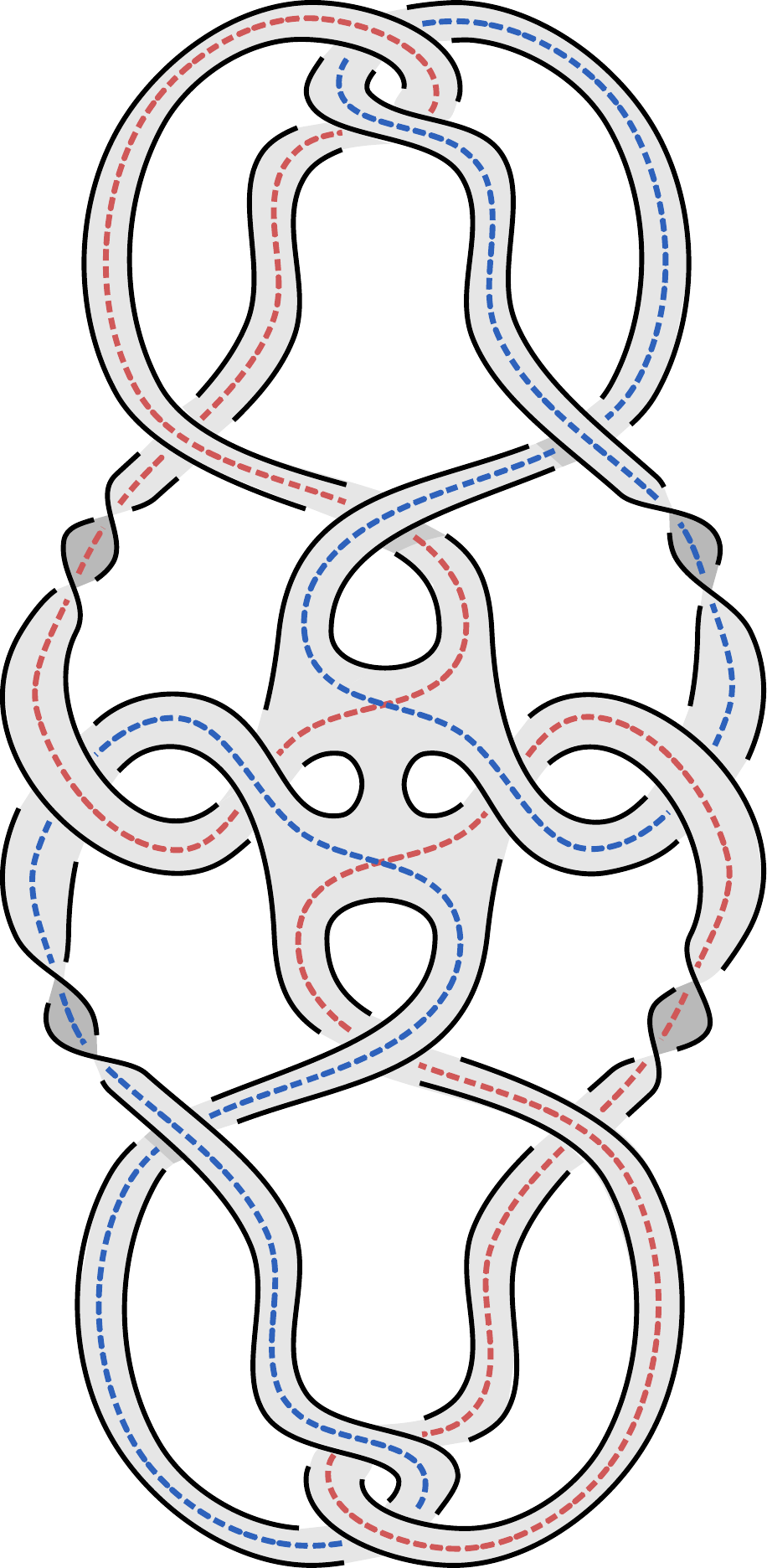}

\vspace{-5pt}
\caption{}\label{fig:seifert}
\end{wrapfigure}

\

\vspace{-18pt}

The stabilized surface  $\Sigma_{1/2} \subset X$ contains two distinct sets of $n$  disjoint simple closed curves that bound compressing disks in $X \setminus \Sigma_{1/2}$; compressing along one system of curves yields $\Sigma_0$ and the other yields $\Sigma_1$.  To prevent the recycling of tubes, one seeks to ensure that no compressions from one set of curves can be performed if \emph{any} compressions from the other set of curves has been performed.  Explicit examples can be constructed directly by interlocking pairs of surfaces as described in \S\ref{sec:intro}, or by realizing them as branch sets of branched coverings $X\to B^4$ where $X$ is a 4-manifold constructed as in \S\ref{subsec:holding}.

\begin{example} The disks $D$ and $D'$ from Figure~\ref{fig:hooked-disks} are obtained by compressing the genus-2 surface in Figure~\ref{fig:seifert} along the two sets of curves shown. Observe that the natural compressing disk for each red curve has clasp intersections with \emph{both} blue curves' compressing disks, and vice versa, preventing any obvious recycling of tubes. 
\end{example}

\medskip

\section{Input from Bar-Natan homology} 

\smallskip

We refer the reader to \cite{barnatan} for background, but we recall some essential details here. Our discussion mostly follows \cite{kwz:immersed}, including their notation and conventions. Throughout, we will use $\kk$ to denote the field $\kk_2 $ and $H$ to denote a formal variable. Over the ring $\kk[H]$, the Bar-Natan complex $\cbnr(K)$ admits the following structure theorem (where $h^a q^b \,\kk[H]$ denotes a bigraded copy of $\kk[H]$ whose generator $1 = H^0$ sits in bigrading $(a,b)$):

\setlength\extrarowheight{2pt}

 \begin{wrapfigure}[10]{r}[.85cm]{0.31\textwidth} \vspace{6pt}
%\begin{figure}
 \vspace{-2.5pt}\small
 \footnotesize
\hspace{3pt} \begin{tabular}{m{28pt}||m{50pt}c}
 \arrayrulecolor{gray}
  \multicolumn{1}{r||}{ \color{gray}{\backslashbox{\raisebox{2pt}{\!$q$\!}}{\!$h$}} } & \multicolumn{1}{c}{\color{gray} \footnotesize \hspace{7pt} $a$ \hspace{7 pt} \vrule   \hspace{3pt} $a+1$ \hspace{0pt}  }\\
    \hline \hline \
    \raisebox{-3pt}{\color{gray}{\hspace{8pt}\scriptsize$b$}} & \multirow{2}{*}{
         \hspace{-22.5pt}
         \begin{minipage}{2.5cm}\vspace{05pt}
      \begin{tikzcd}[ampersand replacement=\&, row sep = 5pt, column sep = 17.5pt   
    ,/tikz/column 1/.append style={anchor=base east}
    ,/tikz/column 2/.append style={anchor=base west}    ]
             \& \hspace{3pt} \xi \arrow[gray,bend left=60]{d}{\text{$H \cdot$}} \\[.035cm]
            \& H\xi \\[-.25cm]
            \& \hspace{6pt} \vdots   \\[.05cm]
            %\arrow[gray,bend right=60,swap]{d}{\text{$H \cdot$}}
           \hspace{5pt}   \eta \hspace{2pt} \arrow[line width=0.6pt, maps to]{r} \& H^m \xi  \\[0.035cm]%[-.025cm]
             \hspace{5pt}    H\eta \arrow[line width=0.6pt, maps to]{r} \& [xshift = 3 pt] H^{m+1} \xi   
            \\[-.25cm]
          \vdots  \hspace{6pt} \&  \hspace{6pt} \vdots   
        \end{tikzcd}
        \end{minipage}
    }\\[.25cm]
    \cline{1-1}
     \raisebox{-3.5pt}{\scriptsize\color{gray}{\hspace{3pt} $b\!-\!2$}}  & 
     \\[.25cm]
    \cline{1-1}
    {\color{gray} \raisebox{-3.5pt}{\hspace{9pt} \vdots}} & 
     \\[.25cm]
    \cline{1-1}
   { \color{gray} \raisebox{-3.5pt}{\hspace{1 pt} \scriptsize$b\!-\!2m$}} & 
     \\[.25cm]
    \cline{1-1}  
        {\color{gray}  \raisebox{-3.5pt}{\hspace{-7pt} \scriptsize$b\!-\!2m\!-\!2$}} & 
     \\[.25cm] 
     \cline{1-1}{\color{gray} \hspace{9pt} \vdots}
\end{tabular}
\vspace{-3pt}
\caption{The  map in \eqref{eq:torsion}. \ }\label{fig:tower-map}
\end{wrapfigure}
%\leavevmode
%\end{figure}
\color{black}

\

\vspace{-22.5pt}

\smallskip

\begin{theorem}\label{thm:structure}
For any knot $K \subset S^3$ and any field $\kk$, the complex $\cbnr(K)$ is bigraded chain homotopic to a complex of the form

\begin{minipage}{\linewidth}
\begin{equation}\label{eq:torsion}
h^0 q^{s}\, \kk[H] \, \bigoplus_i \bigg[ h^{a_i-1} q^{b_i-2m_i} \kk[H] \xlongrightarrow{\ H^{m_i} \cdot \ }  h^{a_i} q^{b_i} \kk[H]\bigg]
\end{equation}
\end{minipage}

%\medskip

\noindent for uniquely determined $a_i, m_i \in \Z$ with $m_i>0$ and  $s, b_i  \in 2\Z$. 

The differential vanishes on the tower $h^0q^s \kk[H]$ and is given as in Figure~\ref{fig:tower-map} on the remaining summands (where $\eta$ and $\xi$ represent generators of $h^{a-1} q^{b-2m} \kk[H]$ and $h^{a} q^{b} \kk[H]$, respectively).

In particular, $\bnr(K)$ decomposes as the direct sum of a tower $h^0 q^s \, \kk[H]$ and $H$-torsion summands $h^{a_i} q^{b_i} \kk[H]/(H^{m_i})$.
\end{theorem}

\medskip

The quotient of $\cbnr(K)$ by the subcomplex $H  \cdot  \cbnr(K)$ (i.e., the result of setting $H=0$) yields a complex with trivial differential whose homology yields $\khr(K)$. 

\subsection*{Cobordism maps} Our arguments  rely on the functoriality of Bar-Natan's theory: An oriented link cobordism $\Sigma: K_0 \to K_1$ in $S^3 \times [0,1]$ induces a bigraded chain map
%\medskip

\begin{minipage}{\linewidth}
$$\cbnr(\Sigma): \cbnr(K_0)_{(h,q)} \to \cbnr(K_1)_{(h,q+\chi(\Sigma))}$$
\end{minipage}

\smallskip

%\medskip

\noindent that induces a homomorphism $\bnr(\Sigma): \bnr(K_0) \to \bnr(K_1)$  \cite{barnatan}. The map $\bnr(\Sigma)$ is invariant up to multiplication by $\pm1$ under diffeomorphisms of $S^3 \! \times \! [0,1]$  fixing $\partial \Sigma$ setwise (\cite{barnatan,morrison-walker-wedrich}). Indeed, the  homotopy type of the chain map $\cbnr(\Sigma)$ is itself invariant (up to sign), so we obtain a commutative diagram

\vspace{-22pt}
\begin{equation}\label{eq:project}
\begin{tikzcd}
\cbnr(K_0) \arrow{r}{\cbnr(\Sigma)}   \arrow{d}{\pi} &[1cm] \cbnr(K_1) \arrow{d}{\pi}\\[.125cm]%
\ckhr(K_0)  \arrow{r}{\ckhr(\Sigma)}  & \ckhr(K_1) 
\end{tikzcd}
%\qquad
% \begin{tikzcd}
%\ckhl(K) \arrow{r}{\pi}  & \ckh(K) \\[.5cm]%
%\ckhl(U) \arrow{u}{\ckhl(D)}  \arrow{r}{\pi} & \ckh(U) \arrow[swap]{u}{\ckh(D)}
%\end{tikzcd}
\end{equation}

\vspace{-8pt}

%\smallskip
\vspace{-1pt}

\noindent which induces a corresponding commutative diagram at the level of homology.  We will not make direct calculations with the cobordism maps in Bar-Natan's theory; instead, we will perform calculations at the level of Khovanov homology and lift our arguments to Bar-Natan's theory where needed. The maps on $\ckhr$ induced by Bar-Natan's theory coincide with the maps in \cite{jacobsson}.  As an additional reference for these chain-level maps, we refer the reader to the tables in \cite{hayden-sundberg}.

\begin{remark}\label{rem:puncture}
While the reduced theory offers some conveniences, it does not directly assign maps to  cobordisms $\emptyset \to K$ or $K \to \emptyset$. We will follow the standard remedy: Let $K \subset S^3$ be a knot bounding a smooth surface $\Sigma \subset B^4$. Choose any point $p \in \mathring{\Sigma}$ and fix a small 4-ball $V$ centered at $p$ such that the pair $(V,\Sigma \cap V)$ is diffeomorphic to the standard pair $(B^4,D^2)$ bounded by $(S^3,U)$. Fixing an identification between $B^4 \setminus \mathring{V}$  and $S^3  \! \times\!  [0,1]$ (restricting to the identity on $\partial B^4 = S^3 \! \times \! \{1\}$) gives rise to an associated punctured cobordism $\Sigma \setminus \mathring{D}^2$ in $S^3  \! \times\!  [0,1]$ from the unknot $U \subset S^3  \! \times\!  \{0\}$ to the knot $K \subset S^3 \! \times\! \{1\}$.  We let $\ckhr(\Sigma)$ denote the associated cobordism map
$\ckhr(\Sigma \setminus \mathring{D}^2) : \ckhr(U) \to \ckhr(K).$

While different choices in the construction may yield different punctured cobordisms in $S^3  \! \times\!  [0,1]$ from $U$ to $K$, these  will agree up to diffeomorphisms of $S^3  \! \times\!  [0,1]$ fixing $S^3  \! \times\!  \{1\}$ pointwise. The diffeomorphism of  $S^3  \! \times\!  \{0\}$ may be nontrivial but is isotopic to the identity by Cerf's theorem; thus its induced  map differs by precomposing with an automorphism of $\khr(U)$ induced by an isotopy of $S^3$ taking $U$ back to itself, hence by  $\pm \mathbbm{1}$.
\end{remark}
\vspace{-1pt}

Finally, we recall the effect of internal stabilization on these cobordism maps:

\vspace{-3pt}

\begin{lemma}[{cf \cite[Corollary~2.3]{alishahi:unknotting}, \cite[Proposition~6.11]{lipshitz-sarkar:mixed}}]\label{lem:stab}
Given a cobordism  $\Sigma$ from $K_0$  to $K_1$, let $\Sigma'$ be the cobordism obtained by stabilizing $\Sigma$ using an arc in $S^3 \times [0,1]$ whose endpoints lie on a single component of $\Sigma$ (and whose interior is disjoint from $\Sigma$). The map induced by $\Sigma'$ satisfies $\bnr(\Sigma')=(2x-H)\cdot \bnr(\Sigma) \equiv H \cdot \bnr(\Sigma)$ over $\F_2[H]$. 
\end{lemma}

\medskip

\section{Distinguishing surfaces after internal stabilization}

\smallskip

\subsection{The underlying disks.} We now return to the disks $D$ and $D'$ bounded by the knot $K$ from from Figure~\ref{fig:hooked-disks} in \S\ref{sec:intro}. For convenience, we work with a slightly simplified diagram for $K$ shown on the left side of Figure~\ref{fig:main-phi}, where the blue bands correspond to the disk $D$ and the green bands correspond to $D'$. (If desired, the reader may take this to be the new definition of $K$ and its slice disks.) 

\begin{figure}[b]
\center

\vspace{-1pt}

\def\svgwidth{.78\linewidth}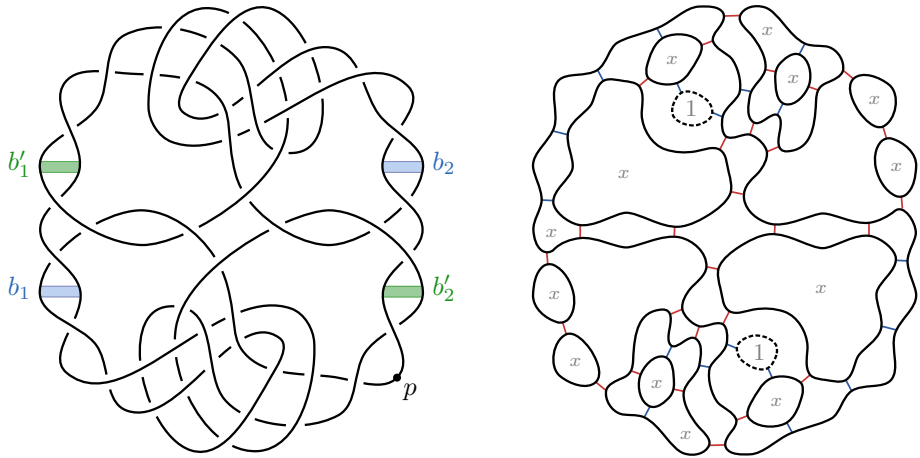

\vspace{-5pt}

\captionsetup{width=.925\linewidth}
\caption{A simplified diagram of the knot $K$, together with a cycle $\phi \in \khr(K)$.}\label{fig:main-phi}
\vspace{-23pt}

\end{figure}

\vspace{-3pt}

\begin{proposition} \label{prop:vanilla}
The disks $D$ and $D'$ induce distinct maps 
$\, \bnr(D)  \neq \bnr(D')$.
\end{proposition}

\vspace{-8pt}

\begin{proof}
Begin by removing a small neighborhood of the absolute minimum of $D \subset B^4$ to obtain a concordance in $S^3 \times [0,1]$ from $-K$ to $-U$ as discussed in Remark~\ref{rem:puncture}. It has a link cobordism movie that begins with a pair of saddle moves on $K$ using the blue bands $b_1$ and $b_2$ in  Figure~\ref{fig:main-phi}. The full link cobordism movie has its key steps illustrated in the first and third rows of Figure~\ref{fig:main-phi-calc}. (Note that the reduced theory requires a choice of basepoint, which we denote by $p$ in Figure~\ref{fig:main-phi} and which is fixed throughout the concordance.)

The righthand side of Figure~\ref{fig:main-phi} depicts a labeled smoothing of the diagram that is easily checked to be a cycle, because all 0-resolution arcs (shown in red) join distinct $x$-labeled circles. Denote the resulting homology class by $\phi \in \khr(K)$. Figure~\ref{fig:main-phi-calc} tracks the image of $\phi$ under the map induced by the concordance $K \to U$ associated to the disk $D$, showing that $\phi$ is mapped to a generator of $\khr(U) \cong \F_2$.

\begin{figure}
\center
\def\svgwidth{\linewidth}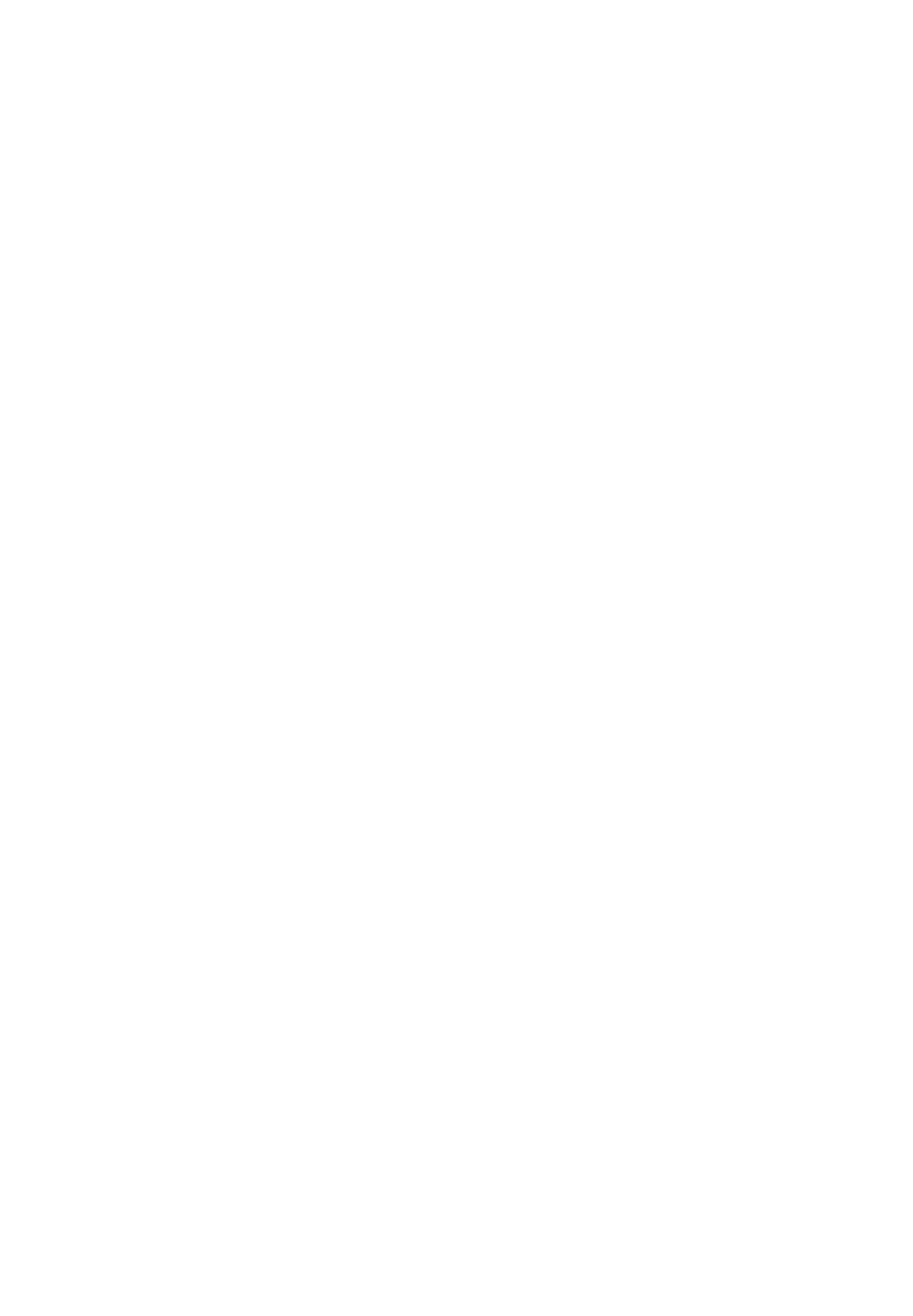

\medskip

\caption{Tracking the element $\phi$ under the map induced by the disk $D$.}\label{fig:main-phi-calc}
\end{figure}

\clearpage
On the other hand, the movie for the concordance associated to $D'$ begins with a pair of saddle corresponding to the bands $b_1'$ and $b_2'$ in Figure~\ref{fig:main-phi}. These saddle moves induce chain maps that merge distinct $x$-labeled circles in $\phi$, hence $\khr(D')$ maps $\phi$ to zero. \end{proof}

We now lift this result to Bar-Natan's theory, where we may deduce more by considering  the role of nontrivial $H$-torsion:

\begin{proposition} \label{prop:hooked-disks}
The stabilized disks induce distinct maps  
$\, H\cdot \bnr(D)  \neq H \cdot \bnr(D' )$.
\end{proposition}

\begin{proof}
For technical convenience, we will work with  the mirrored, time-reversed disks $-D$ and $-D'$ bounded by $-K$. (Diagrammatically, we take this reflection to be through the plane of the page.)  This will be sufficient to prove the claim, since the maps induced by these disks are dual to the ones induced by $D$ and $D'$ (in an appropriate sense, as discussed in \cite[\S7.3]{khovanov00} for Khovanov's original theory or \cite[Lemma~6.5]{lipshitz-sarkar:mixed} for the Bar~Natan setting). In particular, by Proposition~\ref{prop:vanilla}, we have $\khr(-D)\neq \khr(-D')$.

We lift this to Bar-Natan's theory. To that end, let $\xi\approx H^0$ denote the generator of $\bnr(-U) \cong \F_2[H]$. Since $\pi_*(\xi)$ generates $\khr(-U)$, Proposition~\ref{prop:vanilla} implies  $$\khr(-D)(\pi_*(\xi))-\khr(-D')(\pi_*(\xi))\neq 0.$$ The commutative diagram on homology induced by \eqref{eq:project} then implies 
$$\pi_*\left(\bnr(-D)(\xi)-\bnr(-D')(\xi)\right)\neq 0 \implies  \delta := \bnr(-D)(\xi)-\bnr(-D')(\xi) \neq 0.$$
\noindent By Lemma~\ref{lem:stab}, stabilization has the effect of multiplication by $H$. Therefore, to distinguish the maps induced by these stabilized disks, it suffices to show that $H\cdot \delta$ is nonzero.  
  To that end, we use the program \textsl{KnotJob} \cite{knotjob}, to compute $\khr(-K)$ and the Bar-Natan--Lee--Turner spectral sequence,  entering $-K$ using its Dowker-Thistlethwaite code:
    \begin{center}\tt \footnotesize
  -28 -84 14 -88 -16 4 -24 86 -78 -76 44 82 -12 -2 -92 -64 54 -52 72 -56 -34 68 20 \\ -74 -38 60 -42 -62 50 -70 40 -32 -30 90 36 -58 -48 -46 -18 8 -6 26 -10 -80 22 66
  \end{center}

%planar diagram code:
%\begin{center}\tt \footnotesize
%(59,87,60,86),(39,15,40,14),(38,4,39,3),(19,9,20,8),(26,18,27,17),(1,18,2,19),\\
%\noindent(20,91,21,92),(37,24,38,25),(52,73,53,74),(83,70,84,71),(34,9,35,10),(87,59,88,58),\\
%\noindent(84,50,85,49),(62,70,63,69),(85,61,86,60),(43,77,44,76),(48,82,49,81),(25,36,26,37),\\
%\noindent(78,11,79,12),(65,55,66,54),(41,13,42,12),(31,89,32,88),(21,29,22,28),(72,64,73,63),\\
%\noindent(6,27,7,28),(7,1,8,92),(80,55,81,56),(47,64,48,65),(13,41,14,40),(15,4,16,5),\\
%\noindent(66,45,67,46),(67,75,68,74),(51,69,52,68),(56,33,57,34),(75,45,76,44),(71,82,72,83),\\
%\noindent(77,43,78,42),(89,31,90,30),(10,79,11,80),(32,57,33,58),(53,47,54,46),(61,50,62,51),\\
%\noindent(29,91,30,90),(5,23,6,22),(16,24,17,23),(2,36,3,35)
%\end{center}

The relevant portion of the spectral sequence is shown in  Table~\ref{table:main-ss}. Observe that  $\khr(-K)$ has rank two in bigrading $(0,0)$ --- indeed, it is generated by the elements $\khr(-D)(\pi_*(\xi))$ and $\khr(-D')(\pi_*(\xi))$. All nonzero elements in bigrading $(0,0)$ survive to the second page of the spectral sequence, and we can see that   $\bnr(-K)$ must have the form $h^0q^0\, \F_2[H] \,  \oplus \,  h^0q^0\, \F_2[H]/H^2 \, \oplus \, M$, where  $M$ is supported in bigradings $(h,q)\neq (0,0)$. It follows that all nonzero elements of $\bnr(-K)$ in bigrading $(0,0)$ survive multiplication by $H$, hence the stabilized disks induce distinct maps on Bar-Natan homology.
\end{proof}

\begin{table}\tiny
\centering
\setlength\extrarowheight{2pt}
\begin{tabular}{|c||cc|>{\centering}m{.02\textwidth}|>{\centering}m{.02\textwidth}|>{\centering}m{.02\textwidth}|>{\centering}m{.02\textwidth}|>{\centering}m{.02\textwidth}|>{\centering}m{.02\textwidth}|c|}
 \arrayrulecolor{black}
\multicolumn{10}{c}{Page 1} \smallskip \\
\cline{1-10}  
\color{black}\backslashbox{\!$q$\!}{\!$h$\!} &  $\hdots$   &\hspace{-8pt} {\color{black}{\vrule}} \hspace{-2pt} \color{black}$-4$ \hspace{-5pt} & \color{black}$-3$ & \color{black}$-2$ & \color{black}$-1$ & \color{black}$0$ & \color{black}$1$ & \color{black}$2$ & \color{black}$3$ \\
\hhline{=||=========}%\cmidrule{1-10}\morecmidrules\cmidrule{1-10}
$2$     &   &   &   &   &   &   &   &   & $2$ \\ 
\hhline{-||~--------}%\hhline{-||~--------}
$0$     &   &   &   &   &   & \cellcolor{cellgray} 2 &   &  3 &   \\
\hhline{-||~--------}
$-2$     &   &   &   &   &   &  2 & $6$ & $ $ &   \\
\hhline{-||~--------}
$-4$     &   &   &   & 2   &   13 & $14$ &   &   &   \\
\hhline{-||~--------}
$-6$     &   &   &    4 &  24 & 19 &  &   &   &   \\
\hhline{-||~--------}
$-8$  &   & 13  & 44   & 24 &  &   &   &   &   \\
\hhline{-||~--------}
$-10$     &   & 75   & 28 &   &   &   &   &   &   \\
\hhline{-||~--------}
$-12$    &   & 26&  &   &   &   &   &   &   \\
\hhline{-||}
$\vdots$ & \  \reflectbox{$\ddots$}
\\
\end{tabular}\hfill
\begin{tabular}{|c||cc|>{\centering}m{.02\textwidth}|>{\centering}m{.02\textwidth}|>{\centering}m{.02\textwidth}|>{\centering}m{.02\textwidth}|>{\centering}m{.02\textwidth}|>{\centering}m{.02\textwidth}|c|}
\multicolumn{10}{c}{}\\
\multicolumn{10}{c}{Page 2} \smallskip \\
\cline{1-10}  
\color{black}\backslashbox{\!$q$\!}{\!$h$\!} &  $\hdots$   &\hspace{-8pt} {\color{black}{\vrule}} \hspace{-2pt} \color{black}$-4$ \hspace{-5pt} & \color{black}$-3$ & \color{black}$-2$ & \color{black}$-1$ & \color{black}$0$ & \color{black}$1$ & \color{black}$2$ & \color{black}$3$ \\
\hhline{=||=========}
$2$     &   &   &   &   &   &   &   &   & \\
\hhline{-||~--------}
$0$     &   &   &   &   &   & \cellcolor{cellgray}  2 &   &   &   \\
\hhline{-||~--------}
$-2$     &   &   &   &   &   &   &   &   &   \\
\hhline{-||~--------}
$-4$     &   &   &   &   &  1 &   &   &   &   \\
\hhline{-||~--------}
$-6$     &   &   &   &   &   &   &   &   &   \\
\hhline{-||~--------}
$-8$     &   &   &   &  &   &   &   &   &   \\
\hhline{-||~--------}
$-10$     &   &   &   &   &   &   &   &   &   \\
\hhline{-||~--------}
$-12$     &   &   &   &   &   &   &   &   &   \\
\hhline{-||}
$\vdots$ & \ \reflectbox{$\ddots$} \\
\multicolumn{10}{c}{}\\
\end{tabular}

\vspace{-10pt}

\captionsetup{width=.95\linewidth}
\caption{The first two pages of the reduced Bar-Natan--Lee--Turner spectral sequence for the knot $K$ from Figure~\ref{fig:main-phi}, shown for  $h \geq -4$ and $q \geq -12$.} 
\label{table:main-ss}
%\vspace{-7.5pt}
\end{table}

\begin{figure}[b]
\center
    \labellist

\pinlabel {\footnotesize (a)} at 220 -25
\pinlabel {\footnotesize (b)} at 845 -25

\endlabellist
\includegraphics[width=.925\linewidth]{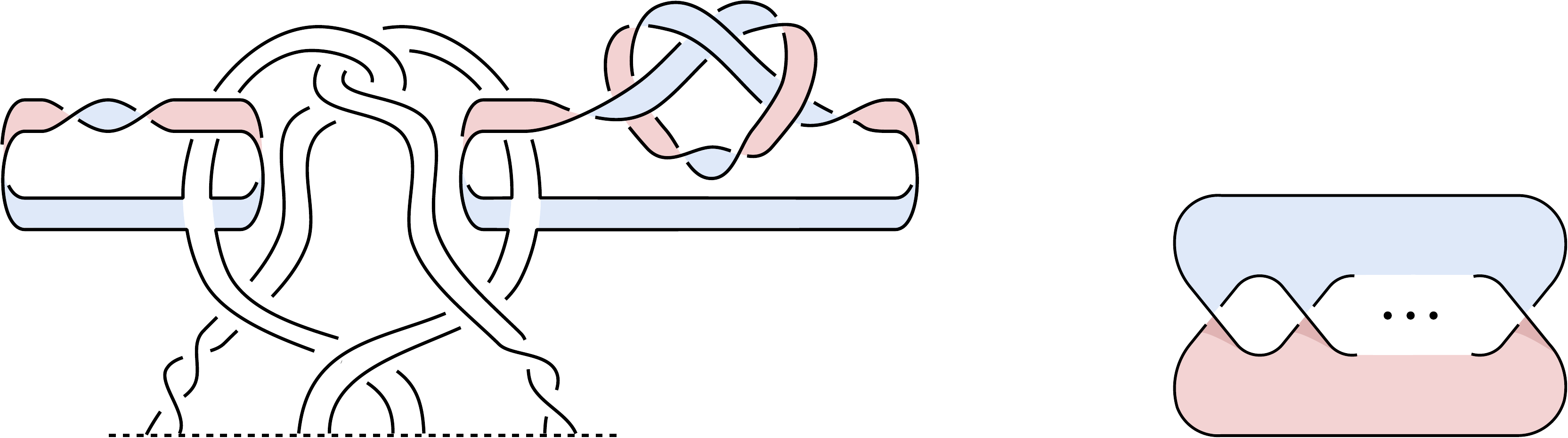}

\vspace{7.5pt}
\captionsetup{width=.925\linewidth}

\caption{(a) Adding a pair of bands to $K$. (b) The standard Seifert surface for  $T_{2,k}$}\label{fig:enlarged}
\vspace{-7.5pt}
\end{figure}

Before proceeding to the proof of Theorem~\ref{thm:stab}, we emphasize that the disks $D$ and $D'$ are \emph{not} topologically isotopic rel boundary. In particular, their associated knot groups $\pi_1(B^4 \setminus D)$ and $\pi_1(B^4 \setminus D')$ are \emph{not} infinite cyclic, and the disks can be distinguished (rel boundary) by the so-called  \emph{peripheral maps} on $\pi_1$ induced by including $S^3 \setminus K$ into the disk complements. Though we expect that variations on our constructions should yield exotic pairs of disks, we content ourselves with producing exotic surfaces of positive genus by attaching bands that eliminate the obstruction to topological isotopy (rel boundary).

Let $F$ and $F'$ denote the genus-1 surfaces in $B^4$ obtained from $D$ and $D'$, respectively, by attaching bands to $K$ as shown in Figure~\ref{fig:enlarged}(a). We will see that $F$ and $F'$ are topologically isotopic rel boundary, yet the obstruction to smooth isotopy persists (even after an internal stabilization). To extend this to an infinite family of surfaces with increasing genus, we set $F_1=F$ and $F_1'=F'$, then let $F_g$ and $F_g'$ be the boundary connected sums of $F_1$ and $F_1'$ with the fiber surface (of genus $g-1$) for the torus knot $T_{2,2g-1}$ depicted in Figure~\ref{fig:enlarged}(b).

 To set convenient notation for the knots that bound these surfaces, we let $K_i = \partial F_i = \partial F_i'$.

\begin{theorem}\label{thm:genus}
Fix an integer $g \geq 1$ and let $F_g,F_g' \subset B^4$ be the surfaces defined above. 
\vspace{-3pt}

\begin{enumerate}[label=\normalfont \text{(\alph*)}]
\item \hypertarget{thm:genus(a)} The surfaces $F_g$ and $F_g'$ induce distinct maps on Bar-Natan homology, and they remain distinct after one internal stabilization.

\vspace{-4pt}

\item \hypertarget{thm:genus(b)}  The surfaces $F_g$ and $F_g'$  are topologically isotopic rel boundary in $B^4$, hence so are the surfaces obtained from them by internal stabilization. 
\end{enumerate}
\end{theorem}

This theorem implies that $F_g$ and $F_g'$ are exotically knotted (rel boundary) and remain so after an internal stabilization, which will establish Theorem~\ref{thm:stab}. The techniques used to establish parts (a) and (b) of Theorem~\ref{thm:genus} are rather different, so we split up the proof.

\begin{proof}[Proof of Theorem~\ref{thm:genus}\hyperlink{thm:genus(a)}{(a)}]
We begin by distinguishing  $\khr(F_g)$ and $\khr(F'_g)$,  which we view as maps $\khr(K_g)\to \khr(U)\cong \F_2$. We will distinguish these using a chosen cycle  $\phi_g \in \khr(K_g)=\khr(K_1 \# T_{2,2g+1})$. This cycle $\phi_g$ is depicted  on the righthand side of Figure~\ref{fig:enlarged-phi}. We note that it lies in bigrading $(0,2g)$. 

The surfaces $F_g$ and $F_g'$ are obtained from $D$ and $D'$ by attaching a genus-$g$ cobordism from $K$ to $K_g$. Viewed from $K_g$ to $K$, this initial cobordism can be expressed through a movie of diagrams beginning with $2g$ saddle moves (corresponding to the bands depicted on the left side of Figure~\ref{fig:enlarged-phi}) followed by a sequence of Reidemeister I and II moves. A straightforward calculation shows that the associated sequence of elementary cobordism maps takes $\phi_g$ to the element $\phi$ from Figure~\ref{fig:main-phi}. Composing this initial cobordism map $\khr(K_g)\to \khr(K)$ with the cobordism maps $\khr(K)\to  \khr(U)\cong \F_2$ induced by $D$ and $D'$, we conclude that $\khr(F_g)(\phi_g)=\khr(D)(\phi)$ generates $\khr(U) \cong \F_2$, whereas $\khr(F_g')(\phi_g)=\khr(D')(\phi)=0$.

\begin{figure}
\center

\def\svgwidth{\linewidth}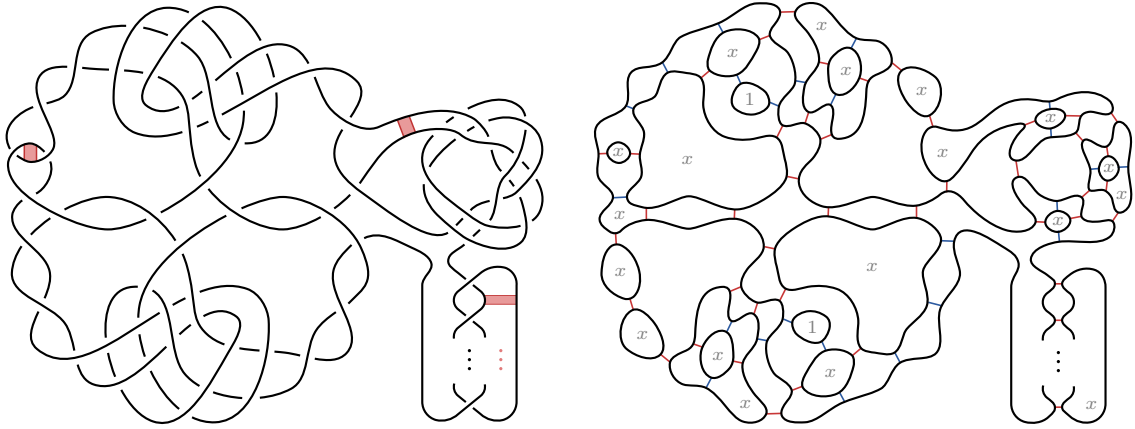
\captionsetup{width=.925\linewidth}

\vspace{1pt}

\caption{A simplified diagram of the knot $K_g$, together with a cycle $\phi_g \in \khr(K_g)$.}\label{fig:enlarged-phi}
\end{figure}

We next lift this to Bar-Natan's theory so that we may compare the maps induced by the stabilized surfaces. Just as in the proof of Proposition~\ref{prop:hooked-disks}, it will be convenient to work with the mirrored, time-reversed surfaces. That is, we consider the surfaces $-F_g$ and $-F_g'$, viewed as cobordisms from $-U$ to $-K_g$. Our argument for distinguishing the cobordism maps induced by the stabilizations of $-F_g$ and $-F_g'$  is formally identical to the analogous argument in Proposition~\ref{prop:hooked-disks} used to distinguish the stabilizations of of $-D$ and $-D'$, so we only highlight the changes.

For $g=1$, the knot $K_1$ has the following Dowker-Thistlethwaite code:

\begin{center}
\begin{minipage}{.93\linewidth}\tt \small
-8,58,-128,66,-4,56,54,52,124,-24,42,-28,122,-20,46,38,36,116,32,30,22,-120,
18,-40,-100,126,14,12,-64,2,-68,10,-60,6,98,96,88,-106,-118,114,112,110,-90,
74,-94,108,-86,104,70,-50,-76,84,-72,92,-102,82,80,34,-78,
-48,26,-44,16,62
\end{minipage}
\end{center}

Using \textsl{KnotJob} \cite{knotjob},  we compute $\khr(-K_1)$ and the Bar-Natan--Lee--Turner spectral sequence. The relevant portion of the spectral sequence is shown in  Table~\ref{table:enlarged}. Observe that  $\khr(-K_1)$ has rank two in bigrading $(0,-2)$ --- indeed, it is spanned by the image of $\khr(-F_1)$ and $\khr(-F_1')$. All nonzero elements in bigrading $(0,-2)$ are seen to survive to the second page of the spectral sequence, so we can see that   $\bnr(-K_1)$ must have the form $h^0q^{-2}\, \F_2[H] \,  \oplus \,  h^0q^{-2}\, \F_2[H]/H^2 \, \oplus \, M$, where  $M$ is supported in bigradings $(h,q) \neq (0,-2)$.  It follows that all nonzero elements of $\bnr(-K_1)$ in bigrading $(0,0)$ survive multiplication by $H$. In particular, we conclude that the stabilized surfaces induce distinct maps.

%PD[X[120,128,121,127],X[128,124,1,123],X[124,120,125,119],X[1,54,2,55],X[53,118,54,119],X[121,50,122,51],X[49,122,50,123],X[48,56,49,55],X[56,52,57,51],X[47,3,48,2],X[57,126,58,127],X[125,58,126,59],X[52,60,53,59],X[3,47,4,46],X[45,5,46,4],X[5,45,6,44],X[43,7,44,6],X[117,43,118,42],X[7,117,8,116],X[115,9,116,8],X[114,36,115,35],X[9,36,10,37],X[113,16,114,17],X[10,16,11,15],X[17,112,18,113],X[18,12,19,11],X[34,112,35,111],X[33,12,34,13],X[14,20,15,19],X[13,32,14,33],X[37,20,38,21],X[38,32,39,31],X[110,39,111,40],X[21,31,22,30],X[29,23,30,22],X[23,29,24,28],X[27,25,28,24],X[89,61,90,60],X[61,89,62,88],X[90,41,91,42],X[40,91,41,92],X[87,63,88,62],X[68,109,69,110],X[92,67,93,68],X[108,69,109,70],X[107,27,108,26],X[25,107,26,106],X[105,71,106,70],X[71,105,72,104],X[103,73,104,72],X[73,103,74,102],X[101,75,102,74],X[100,94,101,93],X[75,94,76,95],X[95,86,96,87],X[96,64,97,63],X[66,98,67,97],X[65,78,66,79],X[79,64,80,65],X[80,86,81,85],X[76,82,77,81],X[84,78,85,77],X[83,98,84,99],X[99,82,100,83]]

\begin{table}\tiny
\centering
\setlength\extrarowheight{2pt}
\begin{tabular}{|c||cc|>{\centering}m{.02\textwidth}|>{\centering}m{.02\textwidth}|>{\centering}m{.02\textwidth}|>{\centering}m{.02\textwidth}|>{\centering}m{.02\textwidth}|>{\centering}m{.02\textwidth}|c|}
 \arrayrulecolor{black}
\multicolumn{10}{c}{Page 1} \smallskip \\
\cline{1-10}  
\color{black}\backslashbox{\!$q$\!}{\!$h$\!} &  $\hdots$   &\hspace{-8pt} {\color{black}{\vrule}} \hspace{-2pt} \color{black}$-4$ \hspace{-5pt} & \color{black}$-3$ & \color{black}$-2$ & \color{black}$-1$ & \color{black}$0$ & \color{black}$1$ & \color{black}$2$ & \color{black}$3$ \\
\hhline{=||=========}%\cmidrule{1-10}\morecmidrules\cmidrule{1-10}
$0$     &   &   &   &   &   &   &   &   & $2$ \\ 
\hhline{-||~--------}%\hhline{-||~--------}
$-2$     &   &   &   &   &   & \cellcolor{cellgray} 2 &   &  6 &  1 \\
\hhline{-||~--------}
$-4$     &   &   &   &   & 3  &  3 & 14 & 3  &   \\
\hhline{-||~--------}
$-6$     &   &   &   & 7    &  18  & 33  &   7 &   &   \\
\hhline{-||~--------}
$-8$     &   &   &   13 &  54 & 70& 16 &   &   &   \\
\hhline{-||~--------}
$-10$  &   & 32  &124   & 129 & 33 &   &   &   &   \\
\hhline{-||~--------}
$-12$     &   & 250   & 219 &  62 &   &   &   &   &   \\
\hhline{-||~--------}
$-14$    &   & 348 & 100  &   &   &   &   &   &   \\
\hhline{-||}
$\vdots$ & \  \reflectbox{$\ddots$}
\\
\end{tabular}\hfill
\begin{tabular}{|c||cc|>{\centering}m{.02\textwidth}|>{\centering}m{.02\textwidth}|>{\centering}m{.02\textwidth}|>{\centering}m{.02\textwidth}|>{\centering}m{.02\textwidth}|>{\centering}m{.02\textwidth}|c|}
\multicolumn{10}{c}{}\\
\multicolumn{10}{c}{Page 2} \smallskip \\
\cline{1-10}  
\color{black}\backslashbox{\!$q$\!}{\!$h$\!} &  $\hdots$   &\hspace{-8pt} {\color{black}{\vrule}} \hspace{-2pt} \color{black}$-4$ \hspace{-5pt} & \color{black}$-3$ & \color{black}$-2$ & \color{black}$-1$ & \color{black}$0$ & \color{black}$1$ & \color{black}$2$ & \color{black}$3$ \\
\hhline{=||=========}
$0$     &   &   &   &   &   &   &   &   & \\
\hhline{-||~--------}
$-2$     &   &   &   &   &   & \cellcolor{cellgray}  2 &   &   &   \\
\hhline{-||~--------}
$-4$     &   &   &   &   &   &   &   &   &   \\
\hhline{-||~--------}
$-6$     &   &   &   &   &  1 &   &   &   &   \\
\hhline{-||~--------}
$-8$     &   &   &   &   &   &   &   &   &   \\
\hhline{-||~--------}
$-10$     &   &   &   &  &   &   &   &   &   \\
\hhline{-||~--------}
$-12$     &   &   &   &   &   &   &   &   &   \\
\hhline{-||~--------}
$-14$     &   &   &   &   &   &   &   &   &   \\
\hhline{-||}
$\vdots$ & \ \reflectbox{$\ddots$} \\
\multicolumn{10}{c}{}\\
\end{tabular}

\vspace{-10pt}

\captionsetup{width=.975\linewidth}
\caption{The first two pages of the reduced Bar-Natan--Lee--Turner spectral sequence for the knot $K$ from Figure~\ref{fig:enlarged-phi}, shown for $h \geq -4$ and $q \geq -14$.}
\label{table:enlarged}
\end{table}

Next we extend this to higher genus ($g\geq 2$), which requires a partial calculation of $\bnr(-K_g)=\bnr(-K_1 \smallsum - \! T_{2,2g-1})$.  Since the torus knots $T_{2,k}$ are alternating, it is simple to compute their reduced homology and associated spectral sequence using the Jones polynomial. In particular, for odd $k>0$, the negative torus knot $-T_{2,k}$ satisfies
$$\bnr(-T_{2,k}) \cong h^0 q^{1-k} \, \F_2[H] \ \boldsymbol{\oplus} \left( \overset{-2}{\underset{i\, =\, -k}{\bigoplus}} h^i q^{-k+2i+1} \, \F_2[H]/H \right).$$

\vspace{-5pt}

Thus, for $k=2g-1$ with $g \geq 2$, we have
$$\bnr(-T_{2,2g-1}) \cong h^0 q^{2-2g} \, \F_2[H] \ \boldsymbol{\oplus} \left( \overset{-2}{\underset{i\, =\, 1-2g}%{\textbf{\Large$\boldsymbol{\oplus}$}}}
{\bigoplus}} h^i q^{2-2g+2i} \, \F_2[H]/H \right)$$

\vspace{-5pt}

Note that $\bnr(-T_{2,2g-1})$ is supported in ``delta-grading'' $d:=q-2h=2-2g$. By the K\"unneth formula for reduced Bar-Natan homology under connected sums, 
we have
$$\bnr^{(0,-2g)}(-K_g) \cong \bigoplus_{j=-\infty}^\infty \left( \bnr^{(j,\,-2g+j)}(-K_1) \otimes_{\F_2} \bnr^{(-j, \, 2-2g-j)}(-T_{2,2g-1}) \right) $$

Since $\bnr(-T_{2,2g-1})$ is supported in bigradings $(h,q)\leq (0,2-2g)$, it is easy to see that the only contributions to the summand of $\bnr(-K_1 \smallsum -\! T_{2,2g-1})$ in bigrading $(h,q)=(0,-2g)$ come from the  summands of $\bnr(-K_1)$ and $\bnr(-T_{2,2g-1})$ in bigradings $(0,-2)$ and $(0,2-2g)$, respectively. 

In particular, the $\F_2$-submodule $\bnr^{(0,-2g)}(-K_g)$ spans a $\F_2[H]$-submodule of $\bnr(-K_g)$ that is isomorphic to \smallskip
$$\text{ $\left(h^0q^{-2}\, \F_2[H]  \oplus h^0q^{-2}\, \F_2[H]/H^2\right) \otimes  h^0 q^{2-2g} \, \F_2[H]  \ \, \cong \, \ h^0 q^{-2g} \, \F_2[H]  \oplus   h^0q^{-2g}\, \F_2[H]/H^2$}$$

It follows that  all nonzero elements in bigrading $(h,q)=(0,-2g)$  survive multiplication by $H$, hence the surfaces remain distinct after a single stabilization.
\end{proof}

To prove Theorem~\ref{thm:genus}\hyperlink{thm:genus(b)}{(b)}, we will need to study the \emph{equivariant intersection forms} associated to our surfaces --- that is, the intersection form on the infinite cyclic cover of the surface exterior in $B^4$. Though we recall some details below,  we point the reader to \cite{milnor:infinite} and \cite{conway-powell} for the relevant background.

To set this up, let $S \subset B^4$ denote a connected, properly embedded surface whose exterior $X = B^4 \setminus \mathring{N}(S)$ has $\pi_1(X) \cong \Z$, and let $\smash{\widetilde X}$ denote the infinite cyclic cover of $X$. The group $H_2(\smash{\widetilde X}; \Z)$ can be viewed as a module over $\Lambda = \Z[t,t^{-1}]$; for convenience, we denote this by $H_2(X; \Lambda)$. The \emph{equivariant intersection form} $\lambda_X$ is the pairing
\vspace{-10pt}

\begin{align*}
\lambda_X: H_2(X; \Lambda) \times H_2(X;\Lambda) \to \Lambda \qquad \qquad
\lambda_X(\alpha,\beta) = \sum_k \left(\alpha \cdot t^k \beta\right) t^k,
\end{align*}

\vspace{-10pt}

where $\alpha \cdot t^k \beta$ is the signed intersection number between the homology class $\alpha$ and the (translated) homology class $t^k \beta$. When convenient, we will sometimes denote this by $\lambda_S$ instead of $\lambda_X$. We will use the following result from \cite[Theorem 7.5.(2)]{conway-powell}:

%\vspace{-2.5pt}

\begin{theorem}[Conway-Powell]\label{thm:cp}
Let $S,S' \subset B^4$ be properly embedded, genus-1 surfaces bounded by a knot with trivial Alexander polynomial $\Delta \equiv 1$. If $\lambda_S$ and $\lambda_{S'}$ are isometric, then $S$ and $S'$ are topologically isotopic rel boundary.
\end{theorem}

%\vspace{-5pt}

\begin{proof}[Proof of Theorem~\ref{thm:genus}\hyperlink{thm:genus(b)}{(b)}]  It suffices to prove that the genus-1 surfaces $F_1$ and $F_1'$ are topologically isotopic rel boundary, since the higher-genus surface  $F_g$ and $F_g'$ with $g \geq 2$ are obtained from $F_1$ and $F_1'$ by attaching a common set of bands to $K_1=\partial F_1=\partial F_1'$.

\begin{figure}
\center
    \labellist

%\pinlabel {(a)} at 305 740
%\pinlabel {(b)} at 937 740
\pinlabel {(a)} at 230 740
\pinlabel {(b)} at 790 740
\pinlabel {(c)} at 140 -15
\pinlabel {(d)} at 570 -15
\pinlabel {(e)} at 1000 -15

\pinlabel {=} at 1071 1348
\pinlabel {=} at 1071 1145

\endlabellist
%\hspace{48pt} \includegraphics[width=.575\linewidth]{enlarged.pdf}
%\hspace{48pt} 
\vspace{10pt}

\includegraphics[width=\linewidth]{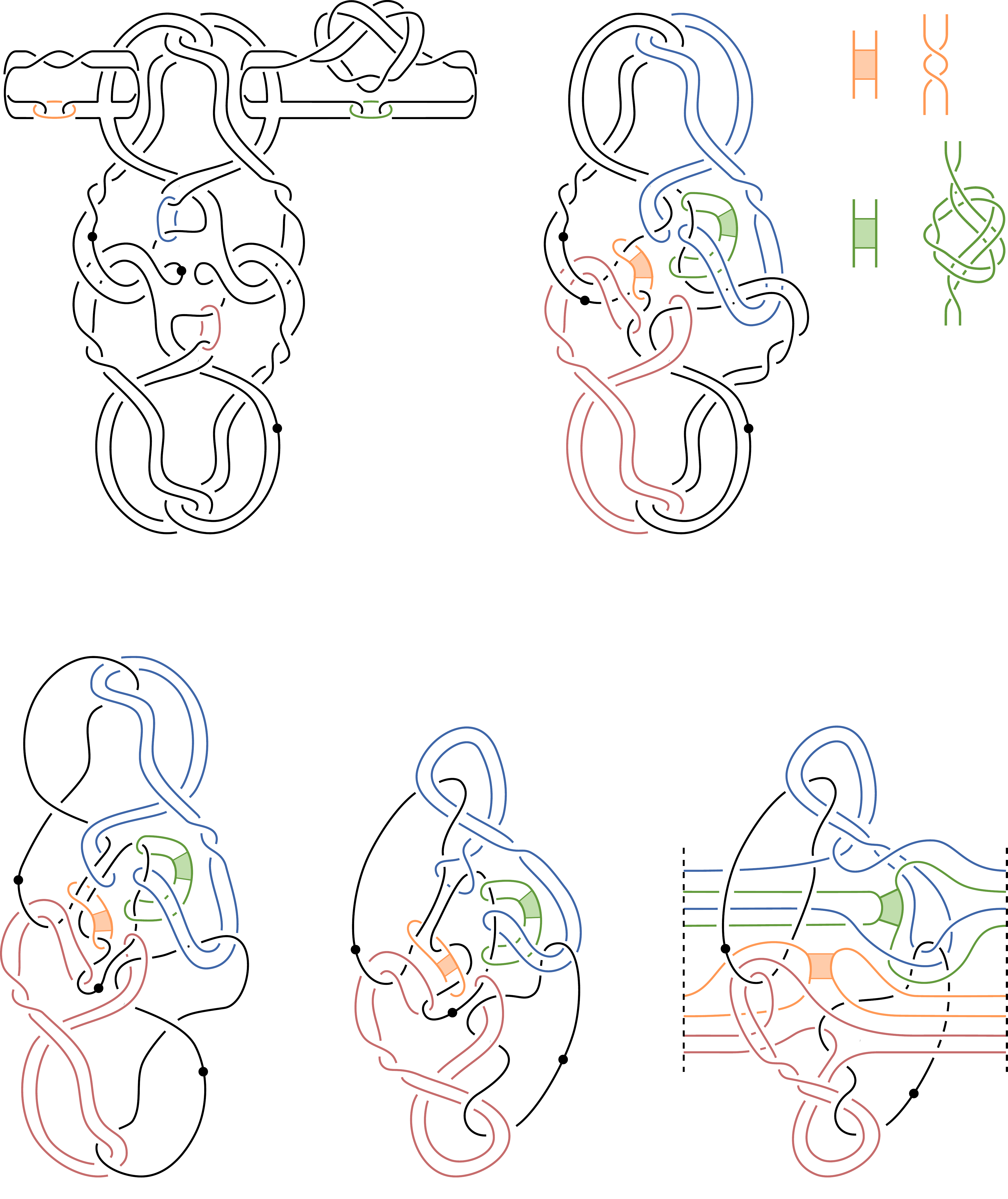}

\vspace{15pt}
\captionsetup{width=.975\linewidth}

\caption{Manipulating handle diagrams for the exterior of $F_1 \subset B^4$. (All undotted curves represent 0-framed 2-handles.) Beginning from part (a), we obtain (b) by simplifying the dotted components at the cost of complicating the 2-handle curves; the orange and green 2-handles have solid boxes where the strands are tied into a negative full twist and a 0-framed trefoil, respectively.  We slide the largest dotted curve over each of the two smaller ones to obtain (c). Further isotopy yields (d). In (e), we  cut open the diagram along the inner dotted component to construct a fundamental domain for the infinite cyclic cover of the surface exterior.}\label{fig:exterior}
\end{figure}

\begin{figure}
\center
\def\svgwidth{\linewidth}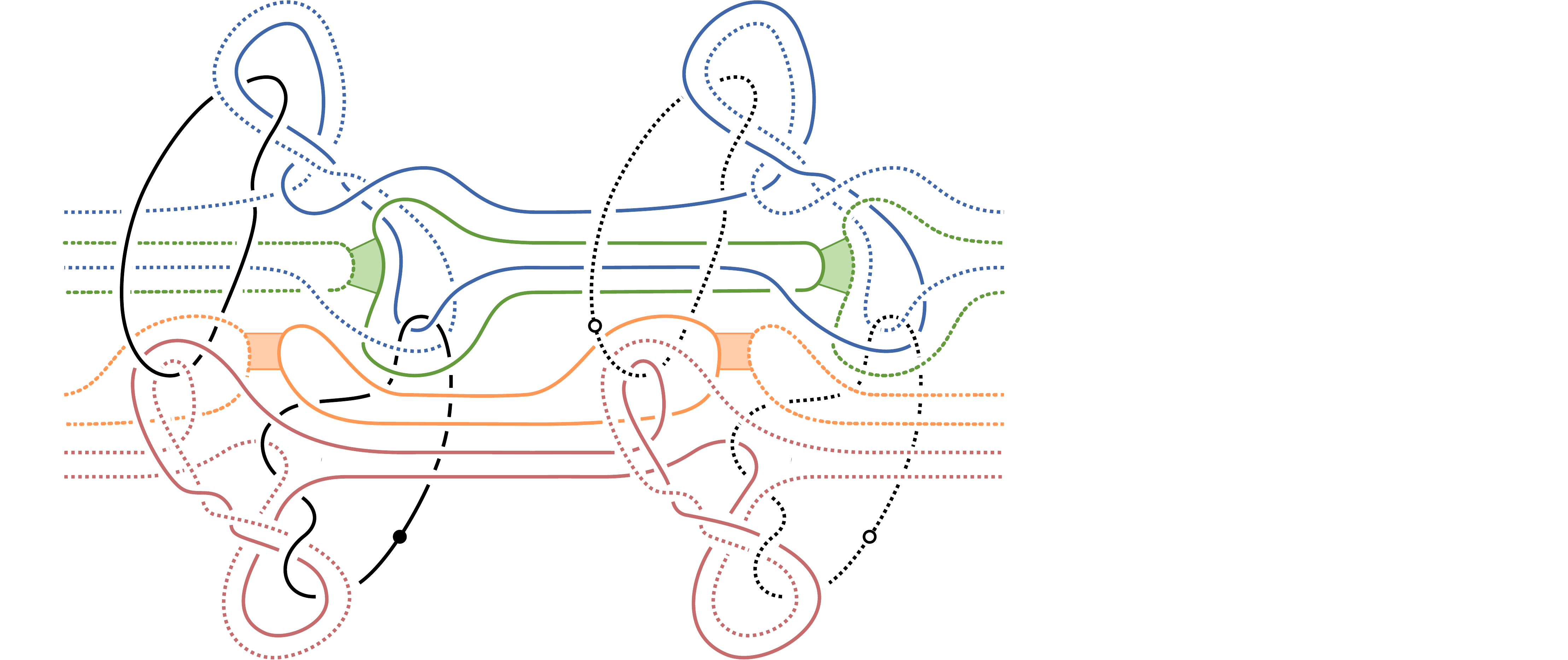
%\captionsetup{width=.9\linewidth}
\caption{A handle diagram representing the infinite cyclic cover of the exterior of $F_1 \subset B^4$. Inside the orange and green boxes, the enclosed strands appear as shown on the righthand side of the diagram.}\label{fig:linking-form}
\end{figure}

 We begin by producing handle diagrams for the surface exterior $X=B^4 \setminus \mathring{N}(F_1)$ and its infinite cyclic cover $\smash{\widetilde X}$ in Figures~\ref{fig:exterior} and \ref{fig:linking-form}. (See the figure captions for some details.) It is straightforward to check that $\pi_1(X) \cong \Z$. The labels in Figure~\ref{fig:linking-form} will be used to describe a basis for $H_2(X;\Lambda)$ and a presentation for $\lambda_X$.
 
As an intermediate step, let $X_0$ denote the 4-manifold obtained from $X$ by surgering the belt spheres of the 1-handles corresponding to $\alpha$ and $\beta$, i.e., treating these as 0-framed 2-handles. Note that this gives rise to a pair of embedded 2-spheres of square zero in $X_0$, and $X$ is recovered from $X_0$ by surgering these spheres. A direct calculation (performed by calculating linking numbers in the diagram) shows that the equivariant intersection form on $H_2(\tilde X_0;\Lambda)\cong \Lambda^6$ is given by the following matrix $Q_0$: 
%\vspace{-5pt}

  \[Q_0 \ \ = \ \ \ 
\begin{blockarray}{ccccccc}
 & \text{\footnotesize\textcolor{gray}{$t^k \alpha$}} & \text{\footnotesize\textcolor{gray}{$t^k \beta$}} & \text{\footnotesize\textcolor{gray}{$t^k \gamma$}} & \text{\footnotesize\textcolor{gray}{$t^k \delta$}} & \text{\footnotesize\textcolor{gray}{$t^k \zeta$}} & \text{\footnotesize\textcolor{gray}{$t^k \eta$}} \smallskip \\ 
\begin{block}{c  [ cccccc]}
    \text{\footnotesize\textcolor{gray}{$\alpha$}} & \
     0 & 0  & \  -t^{-1}+ 1 \ &  1  &  0  & t^{-1}  \
    \\ 
    \text{\footnotesize\textcolor{gray}{$\beta$}} & \
      0 & 0  &  -1 & t^{-1}-1  &  0  & 0 \
    \\
    \text{\footnotesize\textcolor{gray}{$\gamma$}} & \
      1 -t & -1  & 0  & 0  & 0   & 0  \
  \\
  \text{\footnotesize\textcolor{gray}{$\delta$}} & \ 
      1 & -1+t  & 0  & 0  &  1-t  & 0   \
  \\
  \text{\footnotesize\textcolor{gray}{$\zeta$}} & \
      0 &  0 & 0  &  -t^{-1}+1 &    (t^{-1} -2+ t) &   0 \
  \\
  \text{\footnotesize\textcolor{gray}{$\eta$}} & \
      t & 0   & 0   & 0  &    0 & (0)  \ \smallskip
  \\
\end{block}
\end{blockarray}
 \]

\newpage

We change basis by replacing $\gamma$, $\delta$, $\zeta$, and $\eta$ with  $\gamma'=-\gamma+(t^{-1}-1)\eta$, \ $\delta'=(-t+1)\gamma-\delta+(2-t)\eta$, $\zeta'=(t-1)\gamma+\delta+\zeta+(t-2)\eta$ and  $\eta'=t^{-1}\eta$, respectively. The equivariant intersection form transforms as $Q_0 \rightsquigarrow P^T Q_0 \overline{P}$, where $P$ is the  change-of-basis matrix 

$$P=\left[
\begin{array}{cccccc}
 1 & 0 & 0 & 0 & 0 & 0 \\
 0 & 1 & 0 & 0 & 0 & 0 \\
 0 & 0 & -1 & 1-t & t-1 & 0 \\
 0 & 0 & 0 & -1 & 1 & 0 \\
 0 & 0 & 0 & 0 & 1 & 0 \\
 0 & 0 & t^{-1}-1 & 2-t & t-2 & t^{-1}\\
\end{array}
\right]$$

and $\overline{P}$ is obtained from $P$ by mapping $t \mapsto t^{-1}$. This yields
\vspace{-5pt}

 \[P^T Q \overline{P} \ \ = \ \  
\begin{blockarray}{ccccccc}
 & \text{\footnotesize\textcolor{gray}{$t^k \alpha$}} & \text{\footnotesize\textcolor{gray}{$t^k \beta$}} & \text{\footnotesize\textcolor{gray}{$t^k \gamma'$}} & \text{\footnotesize\textcolor{gray}{$t^k \delta'$}} & \text{\footnotesize\textcolor{gray}{$t^k \zeta'$}} & \text{\footnotesize\textcolor{gray}{$t^k \eta'$}} \smallskip \\ 
\begin{block}{c  [ cccccc]}
    \text{\footnotesize\textcolor{gray}{$\alpha$}} & \
     0 & 0  & 0  &  0  &  0  &1  \
    \\ 
    \text{\footnotesize\textcolor{gray}{$\beta$}} & \
      0 & 0  &  1 & 0  &  0  & 0 \
    \\
    \text{\footnotesize\textcolor{gray}{$\gamma'$}} & \
      0 & 1  & 0  & 0  & 0   & 0  \
  \\
  \text{\footnotesize\textcolor{gray}{$\delta'$}} & \ 
      0 & 0  & 0  & 0  &  t-1 & 0   \
  \\
  \text{\footnotesize\textcolor{gray}{$\zeta'$}} & \
      0 &  0 & 0  & t^{-1}-1 &  0 &   0 \
  \\
  \text{\footnotesize\textcolor{gray}{$\eta'$}} & \
      1 & 0   & 0   & 0  &    0 & 0  \ \smallskip
  \\
\end{block}
\end{blockarray}
 \]
 
To obtain the equivariant intersection form $Q$ for $\tilde X$ itself, we surger the 2-spheres corresponding to $\alpha$ and $\beta$ (and their translates) in $\tilde X_0$. Observe that the classes $t^k \alpha$ and $t^k \beta$ are algebraically dual to $t^k \eta'$ and $t^k \gamma'$, respectively, and that these respective duals are the only classes that pair nontrivially with $t^k \alpha$ and $t^k \beta$. It follows that $H_2(\tilde X; \Lambda)$ is generated by the remaining classes corresponding to $\delta'$ and $\zeta'$, and that the equivariant intersection form $\lambda_X$ is presented by the matrix
$$Q = \begin{bmatrix} 0 & t-1\,  \\ \, t^{-1} -1 & 0 \end{bmatrix}.$$

Next we consider $F_1'$. Set $X'=B^4 \setminus\mathring{N}(F_1')$ and let $\smash{\widetilde X}'$ denote its infinite cyclic cover. Recall that $F_1$ and $F_1'$ are obtained by attaching an identical set of bands to the original underlying disks $D$ and $D'$, and that $D$ and $D'$ differ by a 180$^\circ$ rotation through a vertical axis. Therefore, up to isotopy, we can obtain $F_1'$ from $D$ by swapping the two bands that comprise $F_1' \setminus \mathring{D}'$ (i.e., the two bands in part (a) of Figure~\ref{fig:enlarged}) and attaching these to $D$. It follows that the handle diagrams for $X$ and $\smash{\widetilde X}$ in Figures~\ref{fig:exterior}-\ref{fig:linking-form} can be converted to handle diagrams for $X'$ and $\smash{\widetilde X}'$ by swapping the contents of the orange and green boxes.  At the level of linking forms, this merely changes the self-linking terms for $\zeta$ and $\eta$. 

To apply this, we consider the analogous auxiliary 4-manifold $X_0'$ obtained by surgering the 1-handles $\alpha$ and $\beta$ in $X'$. Its equivariant intersection form is obtained from $Q_0$ by swapping the last two diagonal entries (namely, $0$ and $t^{-1}-2+t$).  In this case, we change basis (without involving the pair of 2-handles that will be turned into 1-handles) using the matrix
$$P'=  \left[
\begin{array}{cccccc}
 1 & 0 & 0 & 0 & 0 & 0 \\
 0 & 1 & 0 & 0 & 0 & 0 \\
 0 & 0 & -1 & 1-t & 0 & 0 \\
 0 & 0 & 0 & -1 & 0 & 0 \\
 0 & 0 & 0 & 2 t+2t^{-1}-5 & 1 & 0 \\
 0 & 0 & t^{-1}-1 & 2-t & 0 & t^{-1}\\
\end{array}
\right].$$

This produces a basis in which the equivariant intersection form  is presented by the following matrix (which was computed with the aid of Mathematica \cite{mathematica}): \smallskip
$$\left[
\begin{array}{cccccc}
 0 & 0 & 0 & 0 & 0 & 1 \\
 0 & 0 & 1 & 0 & 0 & 0 \\
 0 & 1 & -t^2+4 t-6+\frac{4}{t}-\frac{1}{t^2} & -\frac{1}{t^3}+\frac{5}{t^2}-2 t+7-\frac{9}{t} & 0 & -t^2+3 t-3+\frac{1}{t} \\
 0 & 0 & -t^3+5 t^2-9 t+7-\frac{2}{t} & 0 & t-1 & -t^3+4 t^2-5 t+2 \\
 0 & 0 & 0 & \frac{1}{t}-1 & 0 & 0 \\
 1 & 0 & -3+t+\frac{3}{t}-\frac{1}{t^2} &2 -\frac{1}{t^3}+\frac{4}{t^2}-\frac{5}{t} & 0 & t-2+\frac{1}{t} \\
\end{array}
\right]$$
\smallskip

After surgering the spheres corresponding to the first two basis elements, we obtained 1-handles that are algebraically dual to the third and sixth basis elements. The remaining two basis elements generate $H_2(\tilde X'; \Lambda)$, which thus has equivariant intersection form
\smallskip
$$Q' = \begin{bmatrix} 0 & t-1\,  \\ \, t^{-1} -1 & 0 \end{bmatrix}.$$

Using Theorem~\ref{thm:cp}, we conclude that $F_1$ and $F_1'$ are topologically isotopic rel boundary.
\end{proof}

\smallskip

\section{Further applications and examples}

\smallskip

\subsection{Building candidate 4-manifolds in the closed setting}
One feature of the ``interlocking'' construction from \S\ref{sec:intro} is that it yields 4-manifolds with explicit and well-controlled handle diagrams, which are convenient for using these pairs of 4-manifolds with boundary to produce exotic 4-manifolds in the closed setting.

\begin{figure}[b]
%\vspace{-10pt}

\center
\def\svgwidth{.425\linewidth}%% Creator: Inkscape 1.2 (dc2aeda, 2022-05-15), www.inkscape.org
%% Accompanies image file 'hooked-positron-stein.pdf' (pdf, eps, ps)
\begingroup%
  \makeatletter%
  \providecommand\color[2][]{%
    \errmessage{(Inkscape) Color is used for the text in Inkscape, but the package 'color.sty' is not loaded}%
    \renewcommand\color[2][]{}%
  }%
  \providecommand\transparent[1]{%
    \errmessage{(Inkscape) Transparency is used (non-zero) for the text in Inkscape, but the package 'transparent.sty' is not loaded}%
    \renewcommand\transparent[1]{}%
  }%
  \providecommand\rotatebox[2]{#2}%
  \newcommand*\fsize{\dimexpr\f@size pt\relax}%
  \newcommand*\lineheight[1]{\fontsize{\fsize}{#1\fsize}\selectfont}%
  \ifx\svgwidth\undefined%
    \setlength{\unitlength}{558.57897084bp}%
    \ifx\svgscale\undefined%
      \relax%
    \else%
      \setlength{\unitlength}{\unitlength * \real{\svgscale}}%
    \fi%
  \else%
    \setlength{\unitlength}{\svgwidth}%
  \fi%
  \global\let\svgwidth\undefined%
  \global\let\svgscale\undefined%
  \makeatother%
  \begin{picture}(1,0.74358121)%
    \lineheight{1}%
    \setlength\tabcolsep{0pt}%
    \put(0,0){\includegraphics[width=\unitlength,page=1]{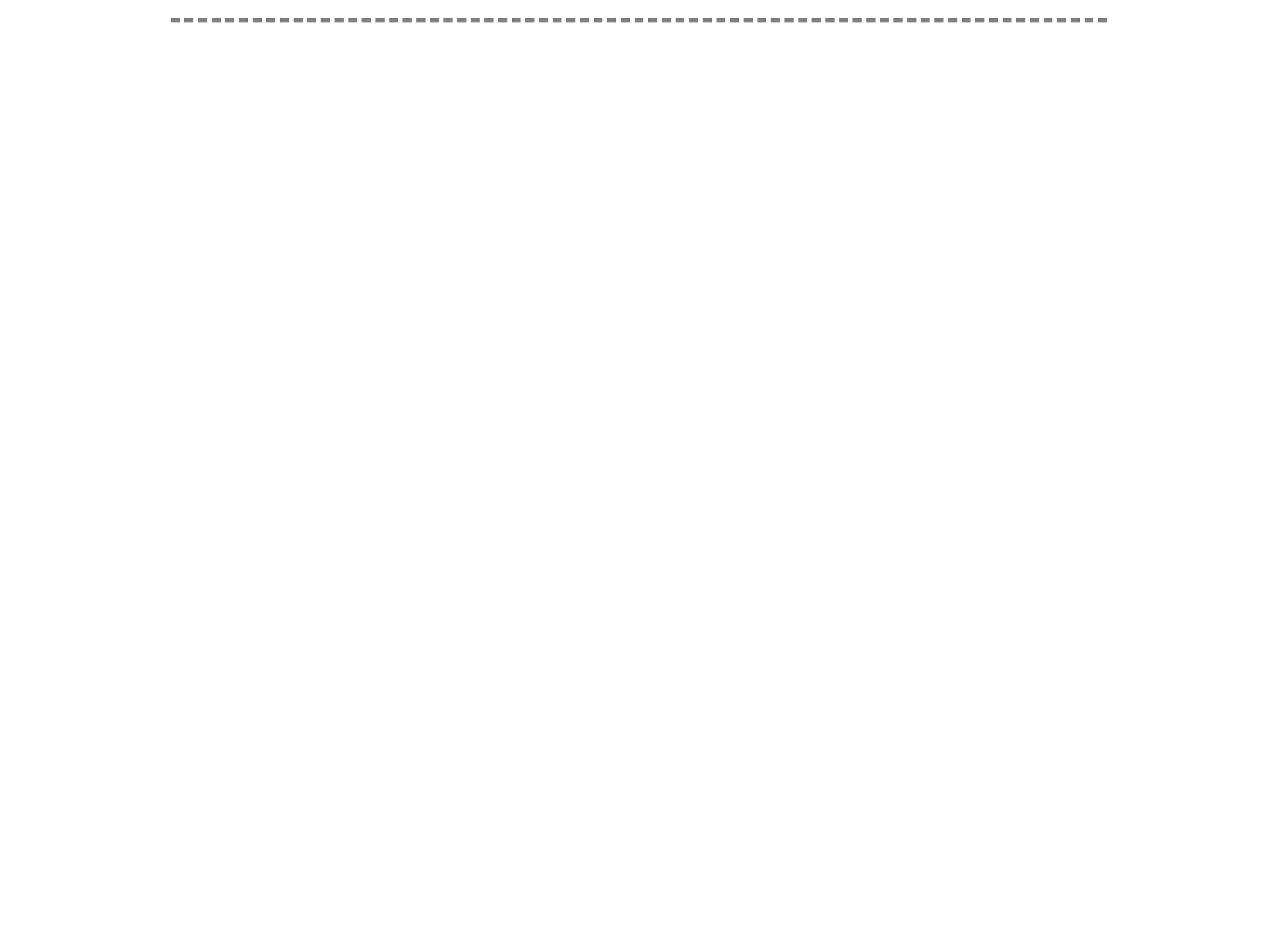}}%
    \put(0.78473796,0.31174249){\color[rgb]{0.24705882,0.40392157,0.6627451}\makebox(0,0)[lt]{\lineheight{1.25}\smash{\begin{tabular}[t]{l}{\footnotesize $0$}\end{tabular}}}}%
    \put(0.19275431,0.40582814){\color[rgb]{0.77647059,0.42352941,0.42352941}\makebox(0,0)[lt]{\lineheight{1.25}\smash{\begin{tabular}[t]{l}{\footnotesize $0$}\end{tabular}}}}%
    \put(0,0){\includegraphics[width=\unitlength,page=2]{hooked-positron-stein.pdf}}%
    \put(0.20055909,0.75432069){\color[rgb]{0.4,0.4,0.4}\makebox(0,0)[lt]{\lineheight{1.25}\smash{\begin{tabular}[t]{l}{\scriptsize $\gamma$}\end{tabular}}}}%
    \put(0,0){\includegraphics[width=\unitlength,page=3]{hooked-positron-stein.pdf}}%
    \put(0.14959823,0.33247324){\color[rgb]{0.4,0.4,0.4}\makebox(0,0)[lt]{\lineheight{1.25}\smash{\begin{tabular}[t]{l}{\scriptsize $\tilde\tau(\gamma)$}\end{tabular}}}}%
  \end{picture}%
\endgroup%
 \hspace{65pt} \raisebox{4pt}{\includegraphics[width=.18\linewidth]{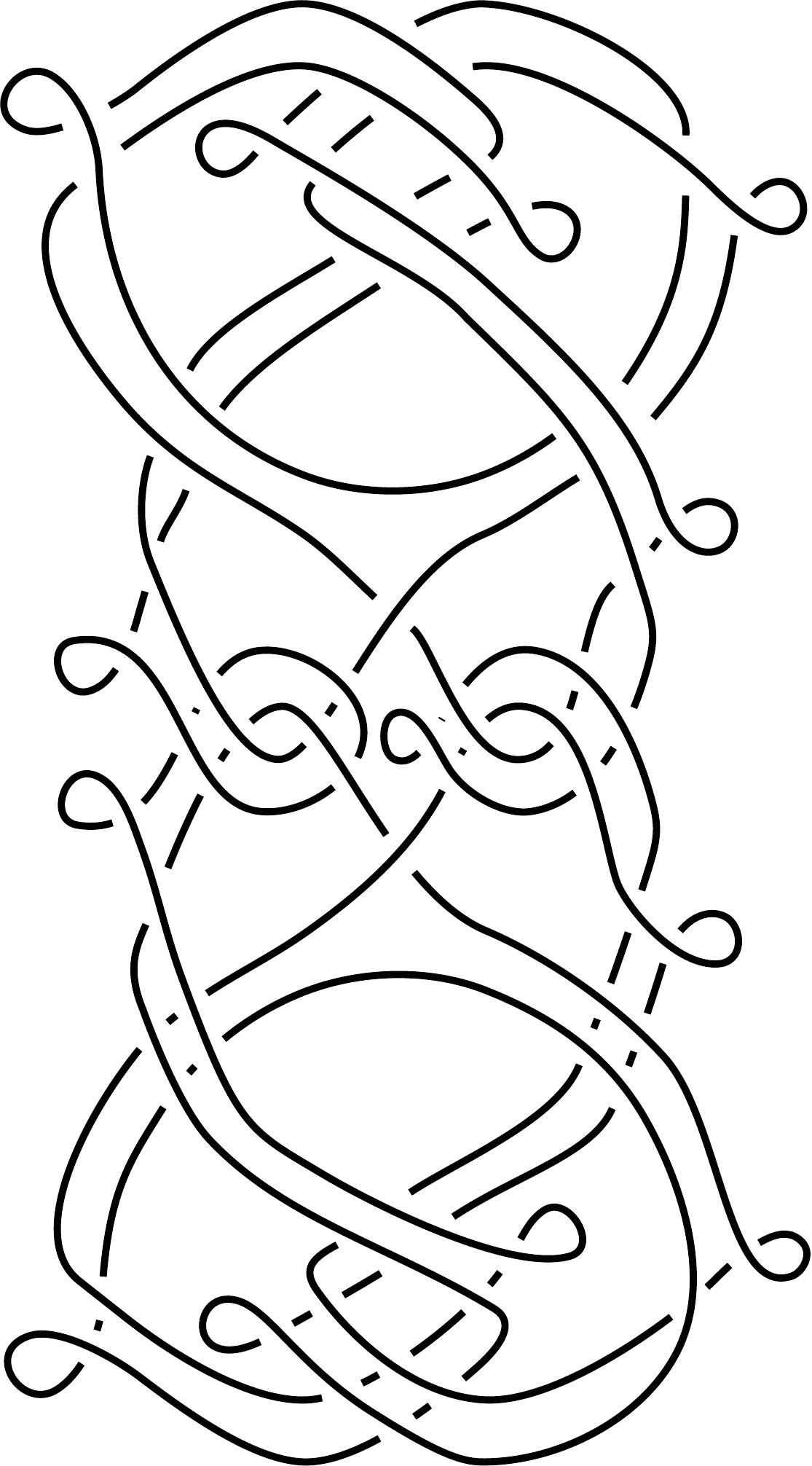}}

\vspace{-3.5pt}

\captionsetup{width=.85\linewidth}

\caption{On the left, a Stein handle diagram for $W=\Sigma_2(B^4,D)$. The  curves $\gamma$ and $\tilde\tau(\gamma)$ are decorations that illustrate the effect of the self-diffeomorphism $\tilde \tau$ on $\partial W$.  \ On the right, a transverse representative $\mathcal{K}$ of the knot $K$ from Example~\ref{ex:disk}.}\label{fig:positron-stein}

\vspace{-22.5pt}

\end{figure}

\vspace{-7.5pt}

\begin{proof}[Proof of Theorem~\ref{thm:embed}]
The left side of  Figure~\ref{fig:positron-stein} depicts a Stein handle diagram for $W$, decorated with cures $\gamma$ and $\tilde\tau(\gamma)$ illustrating the action of $\tilde \tau$ on $\partial W$.  The Stein 4-manifold $W_+$ in Figure~\ref{fig:stein-embed} is obtained from $W$ by attaching two additional Stein 2-handles, shown in green.  (The interested reader can verify that this 4-manifold is the double branched cover of $B^4$ along $F_1$, but we will not need this fact.)  The 4-manifold $W_+'$ on the righthand side of Figure~\ref{fig:stein-embed} is diffeomorphic to the result of cutting out $\mathring{W} \subset W_+$ and regluing it via $\tilde \tau$. To make this more precise, let $L \subset \partial W$ denote the framed link corresponding to the attaching curves of the two additional Stein 2-handles seen on the left side of Figure~\ref{fig:stein-embed}. The union of the 1-handles and the gray 2-handles on the righthand side of Figure~\ref{fig:stein-embed} is naturally identified with $W$, and the green attaching curves in that diagram represent the framed link $\tilde \tau(L) \subset \partial W$. After surgering $W$ along $L$ and $\tilde\tau(L)$, respectively, we obtain $\partial W_+$ and $\partial W_+'$. The diffeomorphism $\tilde \tau$ thus induces a diffeomorphism $f: \partial W_+ \to \partial W_+'$.  

We make two   claims about this map:
\vspace{-2pt}
\begin{enumerate}[label=(\roman*)]
\item the diffeomorphism  $f$ extends to a homeomorphism between $W_+$ and $W_+'$; and
\vspace{-2pt}

\item the diffeomorphism $f$ does not extend to a diffeomorphism between $W_+$ and $W_+'$.
\end{enumerate}

\begin{figure}
\vspace{-10pt}
\center
\def\svgwidth{.85\linewidth}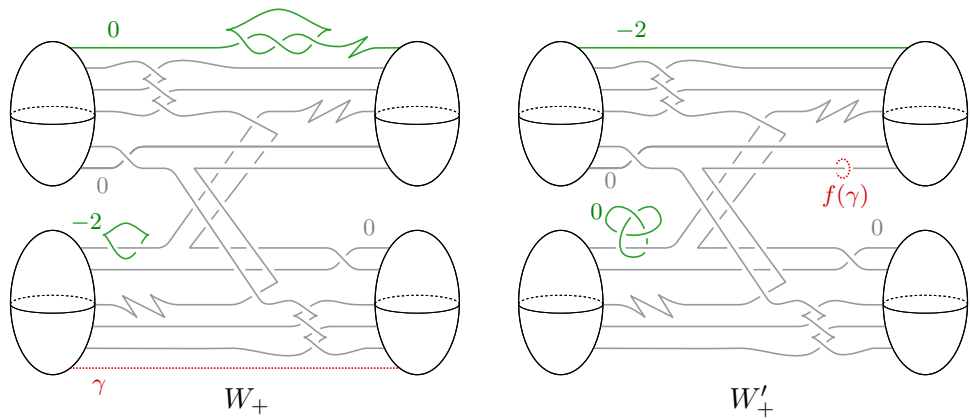 
\caption{Stein handle diagrams, decorated with curves $\gamma$ and $f(\gamma)$.}\label{fig:stein-embed}
\vspace{-5pt}
\end{figure}
\medskip

We begin with (i).  It is straightforward to check that the 3-manifold $\partial W_+ \cong \partial W_+'$ is an integer homology sphere (by hand or using \textsl{SnapPy} \cite{snappy}) and that, with respect to suitable bases, the intersection forms of $W_+$ and $W_+'$ are  given by the standard hyperbolic form $\left[ \begin{smallmatrix} 0 & 1 \\ 1 & 0 \end{smallmatrix}\right]$. 
Therefore, by \cite{freedman} (cf~\cite[\S1]{boyer}), any homeomorphism from $\partial W_+$ to $\partial W_+'$ extends to a homeomorphism from $W_+$ to $W_+'$. Applying this to $f$ yields the first claim.

To prove (ii), we apply the strategy from \cite{akbulut-matveyev}: Consider the 4-manifold obtained by attaching a $(-1)$-framed 2-handle to $W_+$ along the curve $\gamma \subset \partial W_+$ shown in Figure~\ref{fig:stein-embed}. Note that this can be realized as a Stein 2-handle attachment, hence this enlarged 4-manifold admits a Stein structure. On the other hand, $f (\gamma) \subset \partial W_+'$ bounds a smoothly embedded disk in $W_+'$. It follows that attaching a $(-1)$-framed 2-handle along $f(\gamma)$ yields a 4-manifold that contains a smooth 2-sphere of square $-1$. By \cite{lisca-matic:embed}, a Stein domain cannot contain such a sphere, so we conclude that $f$ does not extend to a diffeomorphism from $W_+$ to $W_+'$, proving the second claim.

To prove the claim in the theorem, we again consider the Stein 4-manifold obtained from $W_+$ by attaching a $(-1)$-framed 2-handle along $\gamma \subset \partial W_+$. By \cite{lisca-matic:embed}, this 4-manifold embeds into a minimal complex surface $Z$. Cutting out $W_+ \subset Z$ and gluing in $W_+'$ via $f$ yields a 4-manifold $Z'$ that is homeomorphic to $Z$ by claim (i) above. However, by the argument used to prove claim (ii),  $Z'$ contains a smoothly embedded 2-sphere of square $-1$, hence it cannot be diffeomorphic to $Z$. Furthermore, note that cutting out $W_+$ and gluing in $W_+'$ via $f$ is equivalent to cutting out the underlying copy of $W$ inside $W_+  \subset Z$ and regluing it by $\tilde \tau$. 
\end{proof}

\vspace{-4pt}

With an eye towards invariants from the involutive setting, we ask the following:

\vspace{-3pt}

\begin{question}
Can the closed 4-manifolds in the preceding proof be taken to be spin?
\end{question}

\vspace{2pt}

\subsection{Another pair of disks.} The knot $9_{46}$ bounds a well-known pair of slice disks that induce distinct maps on Khovanov homology \cite{sundberg-swann} (cf \cite{hayden-sundberg}), though the disks themselves are not even topologically isotopic.  Similar to the example from  Figure~\ref{fig:hooked-disks}, we  ``hook'' together two copies of $m(9_{46})$ (and simplify the resulting diagram) to produce the knot shown in Figure~\ref{fig:from946}(a). It bounds a natural pair of slice disks, one of which is shown in part (b) of Figure~\ref{fig:from946}, and other disk is obtained by applying the rotation depicted in part (a) of the figure. The branched cover is shown in Figure~\ref{fig:hooked-946-dbc}.

These two disks induce distinct maps on Khovanov homology (though we note that they are also not even topologically isotopic). Part (c) of Figure~\ref{fig:from946} depicts a class that is easily verified to lie in the support of the map induced by the disk shown and in the kernel of the map induced by the second disk. Just as in the proof of Proposition~\ref{prop:hooked-disks}, these disks induce distinct maps on Bar-Natan homology even after an internal stabilization. In this case, the relevant part of the Bar-Natan--Lee--Turner spectral sequence is shown in Table~\ref{table:946_SS} (computed using \textsl{KnotJob} \cite{knotjob}). \clearpage

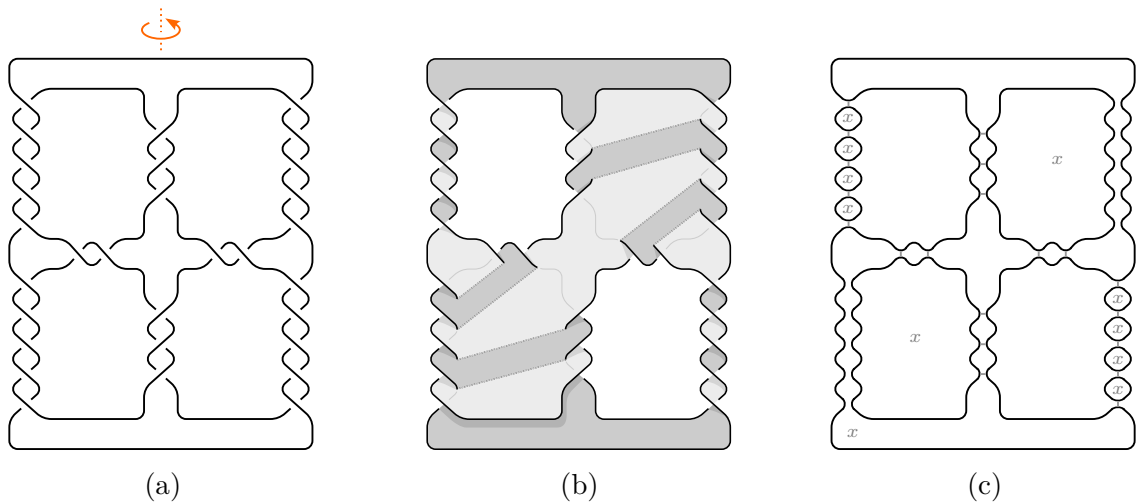
\begin{figure}

\vspace{-11.5pt}
\center
     \labellist
%\pinlabel $K_{9_{46}}$ at 135 -37
%\pinlabel $D_{9_{46}}$ at 508.5 -37
%\pinlabel $\phi \in \khr(-K)$ at 885 -37

\pinlabel \textcolor{gray}{\tiny $x$} at 750 297
\pinlabel \textcolor{gray}{\tiny $x$} at 750 270
\pinlabel \textcolor{gray}{\tiny $x$} at 750 243
\pinlabel \textcolor{gray}{\tiny $x$} at 750 216

\pinlabel \textcolor{gray}{\tiny $x$} at 991 135
\pinlabel \textcolor{gray}{\tiny $x$} at 991 108
\pinlabel \textcolor{gray}{\tiny $x$} at 991 81
\pinlabel \textcolor{gray}{\tiny $x$} at 991 54

\pinlabel \textcolor{gray}{\tiny $x$} at 937 260
\pinlabel \textcolor{gray}{\tiny $x$} at 810 100
\pinlabel \textcolor{gray}{\tiny $x$} at 754 15

%\pinlabel \textcolor{orange}{$\tau$} at 170 378.5

\pinlabel (a) at  138 -33
\pinlabel (b) at  510 -33
\pinlabel (c) at 872 -33
\endlabellist

%\hspace{1pt}
% \includegraphics[width=.95\linewidth]{hooked-946-with-phi}
  \includegraphics[width=\linewidth]{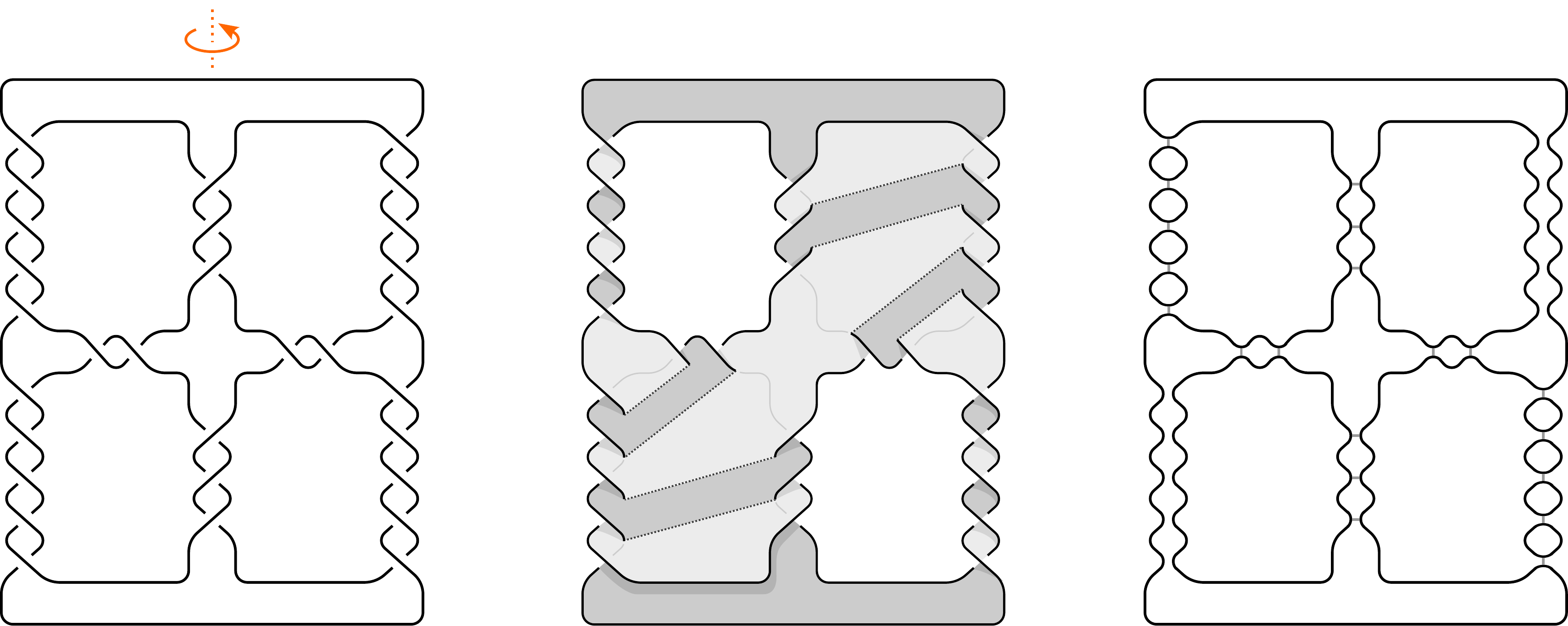}

\bigskip
\medskip

\captionsetup{width=.875\linewidth}

\caption{(a) A knot arising from two copies of $m(9_{46})$. (b) A slice disk bounded by the knot.  (c) An element of Khovanov homology distinguishing the two slice disks.}

\medskip

\label{fig:from946}
\end{figure}

\begin{figure}

%\vspace{-17.5pt}

\bigskip
\center
\def\svgwidth{.975\linewidth}%% Creator: Inkscape 1.2 (dc2aeda, 2022-05-15), www.inkscape.org
%% Accompanies image file 'hooked946dbchoriz.pdf' (pdf, eps, ps)
\begingroup%
  \makeatletter%
  \providecommand\color[2][]{%
    \errmessage{(Inkscape) Color is used for the text in Inkscape, but the package 'color.sty' is not loaded}%
    \renewcommand\color[2][]{}%
  }%
  \providecommand\transparent[1]{%
    \errmessage{(Inkscape) Transparency is used (non-zero) for the text in Inkscape, but the package 'transparent.sty' is not loaded}%
    \renewcommand\transparent[1]{}%
  }%
  \providecommand\rotatebox[2]{#2}%
  \newcommand*\fsize{\dimexpr\f@size pt\relax}%
  \newcommand*\lineheight[1]{\fontsize{\fsize}{#1\fsize}\selectfont}%
  \ifx\svgwidth\undefined%
    \setlength{\unitlength}{932.8942775bp}%
    \ifx\svgscale\undefined%
      \relax%
    \else%
      \setlength{\unitlength}{\unitlength * \real{\svgscale}}%
    \fi%
  \else%
    \setlength{\unitlength}{\svgwidth}%
  \fi%
  \global\let\svgwidth\undefined%
  \global\let\svgscale\undefined%
  \makeatother%
  \begin{picture}(1,0.33879225)%
    \lineheight{1}%
    \setlength\tabcolsep{0pt}%
    \put(0,0){\includegraphics[width=\unitlength,page=1]{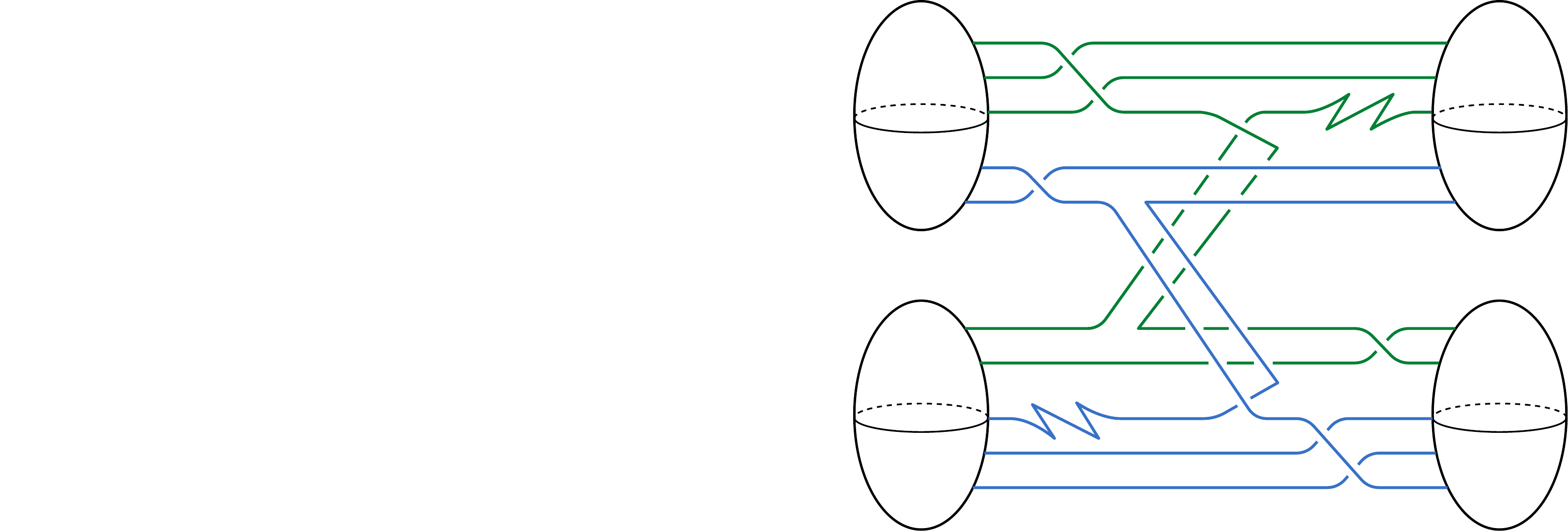}}%
    \put(0.47232742,0.24650254){\color[rgb]{0.21568627,0.44313725,0.78431373}\makebox(0,0)[lt]{\lineheight{1.25}\smash{\begin{tabular}[t]{l}{\small$0$}\end{tabular}}}}%
    \put(0.01967513,0.07401023){\color[rgb]{0,0.50196078,0.2}\makebox(0,0)[lt]{\lineheight{1.25}\smash{\begin{tabular}[t]{l}{\small$0$}\end{tabular}}}}%
    \put(0.63061654,0.14304204){\color[rgb]{0,0.50196078,0.2}\makebox(0,0)[lt]{\lineheight{1.25}\smash{\begin{tabular}[t]{l}{\small $0$}\end{tabular}}}}%
    \put(0.89751456,0.17856576){\color[rgb]{0.21568627,0.44313725,0.78431373}\makebox(0,0)[lt]{\lineheight{1.25}\smash{\begin{tabular}[t]{l}{\small $0$}\end{tabular}}}}%
    \put(0,0){\includegraphics[width=\unitlength,page=2]{hooked946dbchoriz.pdf}}%
  \end{picture}%
\endgroup%

\captionsetup{width=\linewidth}
\caption{The double branched cover of $B^4$ along the disk from Figure~\ref{fig:from946}. The first diagram exhibits two symmetries of its boundary, and the second diagram  exhibits a Stein  structure.  Although it is \emph{not} contractible, the reader can verify that this 4-manifold may still be used to produce exotic pairs of larger 4-manifolds (as in the proof of Theorem~\ref{thm:embed}).
}\label{fig:hooked-946-dbc}

%\vspace{-5pt}
\end{figure}

\begin{table}[b]\tiny
\medskip
\centering
\setlength\extrarowheight{2pt}
\begin{tabular}{|c||cc|>{\centering}m{.02\textwidth}|>{\centering}m{.02\textwidth}|>{\centering}m{.02\textwidth}|>{\centering}m{.02\textwidth}|>{\centering}m{.02\textwidth}|>{\centering}m{.02\textwidth}|c|}
 \arrayrulecolor{black}
\multicolumn{10}{c}{Page 1} \\
\cline{1-10}  
\color{black}\backslashbox{\!$q$\!}{\!$h$\!} &  $\hdots$   &\hspace{-8pt} {\color{black}{\vrule}} \hspace{-2pt} \color{black}$-7$ \hspace{-5pt} & \color{black}$-6$ & \color{black}$-5$ & \color{black}$-4$ & \color{black}$-3$ & \color{black}$-2$ & \color{black}$-1$ & \color{black}$0$ \\
\hhline{=||=========}%\cmidrule{1-10}\morecmidrules\cmidrule{1-10}
$0$     &   &   &   &   &   &   &   &   & \cellcolor{cellgray} $2$ \\ 
\hhline{-||~--------}%\hhline{-||~--------}
$-2$     &   &   &   &   &   &   &   &   &   \\
\hhline{-||~--------}
$-4$     &   &   &   &   &   &   & $2$ & $1$ &   \\
\hhline{-||~--------}
$-6$     &   &   &   &   &   & $4$ &   &   &   \\
\hhline{-||~--------}
$-8$     &   &   &   &   & $9$ & $1$ &   &   &   \\
\hhline{-||~--------}
$-10$  &   &   &   & $17$ & $1$ &   &   &   &   \\
\hhline{-||~--------}
$-12$     &   &   & $21$ &   &   &   &   &   &   \\
\hhline{-||~--------}
$-14$    &   & $27$ & $1$ &   &   &   &   &   &   \\
\hhline{-||}
$\vdots$ & \  \reflectbox{$\ddots$}
\\
\end{tabular}\hfill
\begin{tabular}{|c||cc|>{\centering}m{.02\textwidth}|>{\centering}m{.02\textwidth}|>{\centering}m{.02\textwidth}|>{\centering}m{.02\textwidth}|>{\centering}m{.02\textwidth}|>{\centering}m{.02\textwidth}|c|}
\multicolumn{10}{c}{}\\
\multicolumn{10}{c}{Page 2} \\
\cline{1-10}  
\color{black}\backslashbox{\!$q$\!}{\!$h$\!} &  $\hdots$   &\hspace{-8pt} {\color{black}{\vrule}} \hspace{-2pt} \color{black}$-7$ \hspace{-5pt} & \color{black}$-6$ & \color{black}$-5$ & \color{black}$-4$ & \color{black}$-3$ & \color{black}$-2$ & \color{black}$-1$ & \color{black}$0$ \\
\hhline{=||=========}
$0$     &   &   &   &   &   &   &   &   & \cellcolor{cellgray} $2$ \\
\hhline{-||~--------}
$-2$     &   &   &   &   &   &   &   &   &   \\
\hhline{-||~--------}
$-4$     &   &   &   &   &   &   &   & $1$ &   \\
\hhline{-||~--------}
$-6$     &   &   &   &   &   &   &   &   &   \\
\hhline{-||~--------}
$-8$     &   &   &   &   &   &   &   &   &   \\
\hhline{-||~--------}
$-10$     &   &   &   & $1$ &   &   &   &   &   \\
\hhline{-||~--------}
$-12$     &   &   &   &   &   &   &   &   &   \\
\hhline{-||~--------}
$-14$     &   &   & $1$ &   &   &   &   &   &   \\
\hhline{-||}
$\vdots$ & \ \reflectbox{$\ddots$} \\
\multicolumn{10}{c}{}\\
\end{tabular}

\vspace{-10pt}

\captionsetup{width=.96\linewidth}
\caption{The first two pages of the reduced Bar-Natan--Lee--Turner spectral sequence for the knot $K$ from Figure~\ref{fig:from946}, shown for $h \geq -7$ and $q \geq -14$.}% These results  hold over $\F_2$ and $\Q$.}
\label{table:946_SS}

\vspace{-10pt}

\end{table}

\newgeometry{margin=1.1in,headsep=.3in, top=1.1in, bottom=1.1in,right=1.1in}

\fancyhfoffset[LE,RO]{0cm}
\fancyhfoffset[LO,RE]{-.33cm}

%\bibliographystyle{alpha} 
%\bibliography{biblio}
  \setlength{\bibsep}{4pt plus 0.3ex}

{\small \footnotesize \bibliographystyle{alphamod}
\bibliography{biblio}}

\end{document}